\documentclass{article}

\usepackage{amsmath}
\usepackage{amsfonts}
\usepackage{amssymb}
\usepackage{amsrefs}
\usepackage{amsthm}
\usepackage{graphicx}
\usepackage{hyperref}
\usepackage{upgreek}
\usepackage{mathrsfs}
\usepackage{subcaption}
\usepackage{comment}
\usepackage{cleveref}
\usepackage{tikz}
\usetikzlibrary{math}
\usetikzlibrary{calc}
\usetikzlibrary{arrows.meta}
\usetikzlibrary{backgrounds}
\usepackage{tikz-3dplot}
\pgfdeclarelayer{behind}
\pgfsetlayers{background,behind,main}

\newcommand{\bd}[1]{\operatorname{bd}(#1)}

\newcommand{\vol}[1]{\operatorname{vol}(#1)}
\newcommand{\area}[1]{\operatorname{area}(#1)}

\newcommand{\floor}[1]{\lfloor #1 \rfloor}

\newtheorem{theorem}{Theorem}
\newtheorem{lemma}{Lemma}

\newtheorem{conjecture}{Conjecture}

\theoremstyle{definition}
\newtheorem{definition}{Definition}

\title{The number of touching pairs of congruent sphere packings in Euclidean 3-space}
\author{Cameron Strachan}
\date{}

\begin{document}
\maketitle

\begin{abstract}
    A packing of $n$ congruent balls in $\mathbb{R}^3$ is a family of interior-disjoint Euclidean balls all having the same radius.
    The contact number of a packing is the number of touching pairs of balls.
    In this paper we investigate the problem of determining the maximum contact number, $c(n)$, of a packing of $n$ congruent balls in $\mathbb{R}^3$.
    We first show that all packings of $n$ congruent balls that have a contact number of $c(n)$ are minimally rigid. Furthermore, we show that $c(n)=3n-6$ for $n=6,7,8,$ and $9$.
    These two results resolve a conjecture of K. Bezdek and Khan.
    During the proof of the latter result, we also enumerate the contact structures of all packings of $n$ congruent balls with contact number $c(n)$ for $n=6,7,$ and $8$.
    Additionally, we provide a lower bound construction which shows $c(n)> 6n-6\sqrt[3]{2}n^\frac{2}{3}$ when $n=16k^3-33k^2+24k-6$ where $k\in \mathbb{N}$.
    We also look at the restricted problem where each ball is centered on the face-centered cubic lattice $A_3$. In this case let $c_{A}(n)$ denote the maximum contact number.
    We show that $c_{A}(n)\leq 6n-\frac{6}{\sqrt[6]{2}}n^\frac{2}{3}$ for all $n$, and determine the asymptotics of $c_{A}(n)$ to be $c_{A}(n)=6n-(1+o(1))6\sqrt[3]{2}n^\frac{2}{3}$.
\end{abstract}

\section{Introduction}

Let $B^d$ denote the origin-centered Euclidean ball with unit diameter in $\mathbb{R}^d$.
A finite packing of unit-diameter balls is a finite family of interior-disjoint translates of $B^d$ in $\mathbb{R}^d$.
If the boundaries of two balls intersect, we say the balls touch or are in contact with each other.
Given a packing of $n$ unit-diameter balls, we define an $n$-vertex graph $G$ called the contact graph of the packing.
Each vertex in this graph corresponds to a ball in the packing, and a pair of vertices is connected with an edge if and only if the corresponding balls are in contact with each other.
We can embed $G$ into $\mathbb{R}^d$ by positioning each vertex at the center of the ball it corresponds to.
As the balls form a packing, the minimum distance between any two vertices is $1$, and a pair of vertices has an edge if and only if the distance between them is exactly $1$. 
An embedded graph of this form is called a minimum distance graph.
It is immediate that, up to scaling the set of points so that the minimum distance has unit length, every minimum distance graph in $\mathbb{R}^d$ is the embedded contact graph of a packing of unit-diameter balls in $\mathbb{R}^d$.

The contact number of the packing is the number of pairs of balls that are in contact with each other, or equivalently the number of edges in the contact graph $G$.
The ``\textit{contact number problem}" is to determine, for each $n$, the maximum contact number a packing of $n$ unit-diameter balls in $\mathbb{R}^d$ can have.
We denote this maximum value for each $n$ by $c(n,d)$.
We call a packing of $n$ unit-diameter balls in $\mathbb{R}^d$ and its contact graph extremal if it has a contact number of $c(n,d)$.
Determining $c(n,d)$ is equivalent to Erd\H{o}s's repeated minimum distance problem, which asks for the maximum number of times the minimum distance can occur among $n$ points in $\mathbb{R}^d$ \cites{Erdos46,Erdos75}.
A survey of the problem, its variants, and its motivation from materials sciences can be found in the extensive survey article of K. Bezdek and Khan \cite{BezKhan2018}.

In the plane, Harborth \cite{Harborth74} solved this problem completely and showed $c(n,2)= \floor{3n-\sqrt{12n-3}}$ for all $n\in \mathbb{N}$.
When $n=3k^2-3k+1$ for some $k\geq 1$, a packing that has a contact number of $c(n,2)$ can be constructed by placing the center point of each disk on the unit-length triangular lattice, in such a way that the convex hull of the center points of the disks forms a regular hexagon with a side length of $k-1$ and exactly $k$ center points on each side. 
For all other values of $n$, extremal packings can be obtained by adding disks along the outside of the extremal packing on $3k^2-3k+1$ disks in a spiral fashion.
An example of this is depicted in Figure \ref{Fig: Harborths Examples}.
In particular, the extremal packings can build up from one another.
That is, there exists a sequence of congruent disks in the plane such that for each $n$, the first $n$ disks form a packing with a contact number of $c(n,2)$.

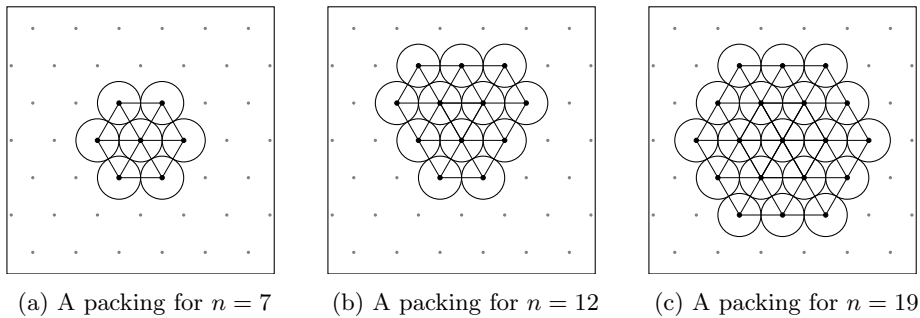
\begin{figure}
\centering
\begin{subfigure}[t]{0.3\textwidth}
\begin{tikzpicture}[>=Stealth,scale=0.57]
\coordinate (e1) at (1,0);
\coordinate (e2) at (60:1);
\coordinate (e3) at (120:1);
\def\points{(0,0),(e1),(e2),(e3), ($-1*(e1)$),($-1*(e2)$),($-1*(e3)$)}
        
        \clip (-3.1,-3.1)--(3.1,-3.1)--(3.1,3.1)--(-3.1,3.1)--cycle;
        \draw (-3.1,-3.1)--(3.1,-3.1)--(3.1,3.1)--(-3.1,3.1)--cycle;
    \foreach \x in {-4,...,4} 
    {
    \foreach \y in {-4,...,4}
    {
    \coordinate (cur) at ($\x*(e1)+\y*(e2)$);
    \filldraw[draw=gray,fill=gray] (cur) circle [radius=0.025];
    }
    }
    \foreach \p in \points{
    \filldraw[draw=black, fill=black] \p circle [radius=0.05];
    \draw[draw=black] \p circle [radius=0.5];
    \draw[very thin,draw=black] (0,0)--\p;
    }
    \draw[very thin,draw=black] (e1)--(e2)--(e3)--($-1*(e1)$)--($-1*(e2)$)--($-1*(e3)$)--cycle;

\end{tikzpicture}
        \caption{A packing for $n=7$}
    \end{subfigure}
    \hfill
    \begin{subfigure}[t]{0.3\textwidth}
\begin{tikzpicture}[>=Stealth,scale=0.57]
\coordinate (e1) at (1,0);
\coordinate (e2) at (60:1);
\coordinate (e3) at (120:1);
\def\points{(0,0),(e1),(e2),(e3), ($-1*(e1)$),($-1*(e2)$),($-1*(e3)$)}
\def\cpoints{(0,0),(e1),(e2),(e3), ($-1*(e1)$),($-1*(e2)$),($-1*(e3)$),($(e1)+(e2)$),($2*(e2)$),($(e3)+(e2)$), ($2*(e3)$), ($(e3)-(e1)$) }
        
        \clip (-3.1,-3.1)--(3.1,-3.1)--(3.1,3.1)--(-3.1,3.1)--cycle;
        \draw (-3.1,-3.1)--(3.1,-3.1)--(3.1,3.1)--(-3.1,3.1)--cycle;
    \foreach \x in {-4,...,4} 
    {
    \foreach \y in {-4,...,4}
    {
    \coordinate (cur) at ($\x*(e1)+\y*(e2)$);
    \filldraw[draw=gray,fill=gray] (cur) circle [radius=0.025];
    }
    }
    \foreach \p in \cpoints{
    \filldraw[draw=black, fill=black] \p circle [radius=0.05];
    \draw[draw=black] \p circle [radius=0.5];
    }
    \foreach \p in \points{
    \draw[very thin,draw=black] (0,0)--\p;
    \draw[very thin,draw=black] (e2)--+\p;
    \draw[very thin,draw=black] (e3)--+\p;
    }
    \draw[very thin,draw=black] ($1*(e1)$)--+(e2)--($2*(e2)$)--($2*(e3)$)--+($-1*(e2)$)--($-1*(e1)$)--($-1*(e2)$)--($-1*(e3)$)--cycle;

\end{tikzpicture}
        \caption{A packing for $n=12$}
    \end{subfigure}
    \hfill
    \begin{subfigure}[t]{0.3\textwidth}
\begin{tikzpicture}[>=Stealth,scale=0.57]
\coordinate (e1) at (1,0);
\coordinate (e2) at (60:1);
\coordinate (e3) at (120:1);
\def\points{(0,0),(e1),(e2),(e3), ($-1*(e1)$),($-1*(e2)$),($-1*(e3)$)}
\def\cpoints{(0,0),(e1),(e2),(e3), ($-1*(e1)$),($-1*(e2)$),($-1*(e3)$),($(e1)+(e2)$),($2*(e2)$),($(e3)+(e2)$), ($2*(e3)$), ($(e3)-(e1)$),($-2*(e1)$),($-2*(e2)$),($-2*(e3)$),($2*(e1)$),($(e1)-(e3)$),($-1*(e3)-(e2)$),($-1*(e2)-(e1)$)}
        
        \clip (-3.1,-3.1)--(3.1,-3.1)--(3.1,3.1)--(-3.1,3.1)--cycle;
        \draw (-3.1,-3.1)--(3.1,-3.1)--(3.1,3.1)--(-3.1,3.1)--cycle;
    \foreach \x in {-4,...,4} 
    {
    \foreach \y in {-4,...,4}
    {
    \coordinate (cur) at ($\x*(e1)+\y*(e2)$);
    \filldraw[draw=gray,fill=gray] (cur) circle [radius=0.025];
    }
    }
    \foreach \p in \cpoints{
    \filldraw[draw=black, fill=black] \p circle [radius=0.05];
    \draw[draw=black] \p circle [radius=0.5];
    }
    \foreach \p in \points{
    \draw[very thin,draw=black] (0,0)--\p;
    \foreach \b in \points 
    \draw[very thin,draw=black] \b--+\p;
    }
    \draw[very thin,draw=black] ($2*(e1)$)--($2*(e2)$)--($2*(e3)$)--($-2*(e1)$)--($-2*(e2)$)--($-2*(e3)$)--cycle;

\end{tikzpicture}
        \caption{A packing for $n=19$}
    \end{subfigure}
    \caption{Three examples of extremal packings with contact number $c(n,2)$ and their contact graphs}
    \label{Fig: Harborths Examples}
\end{figure}

In higher dimensions the contact number problem is far from solved. 
The best general upper bound for $c(n,d)$ comes from a consequence of a result by K. Bezdek in \cite{Bez2002ConvUB}, first noted in \cite{Bez}, which states
\[
c(n,d)< \frac{k(d)}{2}n-\frac{1}{2^d}\delta_d^{-\frac{d-1}{d}}n^\frac{d-1}{d}
\]
where $k(d)$ is the kissing number of the Euclidean ball in $\mathbb{R}^d$ (i.e., the maximum number of interior-disjoint translates of $B^d$ that can touch $B^d$), and $\delta_d$ is the maximum density an infinite packing of congruent balls can have in $\mathbb{R}^d$. When this bound is applied in the case $d=3$, using the results from \cite{SCDW1953Kissing} and \cite{Hales2005} that $k(3)=12$ and $\delta_3=\tfrac{\uppi}{\sqrt{18}}$, we obtain $c(n,3)\leq 6n-\tfrac{1}{8}(\tfrac{\uppi}{\sqrt{18}})^{-\frac{2}{3}}n^\frac{2}{3}=6n-0.152\dots n^\frac{2}{3}$.

For the rest of this paper we will exclusively look at the three-dimensional contact number problem.
As a result we will denote $c(n)=c(n,3)$.
As the smallest dimension in which the contact number problem remains open, the three-dimensional case has received particular attention. 
The first substantial improvement on the upper bound of $c(n)$ stated above is found in \cite{Bez}, which has since been improved by K. Bezdek and Reid in \cite{BezReid2013} to 
\[
c(n)< 6n-0.926n^\frac{2}{3}
\]
for all $n\geq 2$. 
K. Bezdek in \cite{Bez} also found the lower bound, when $n=\frac{k(2k^2+1)}{3}$ for some $k\geq 2$, of 
\begin{align*}
6n-7.862\dots n^\frac{2}{3}= 6n-\sqrt[3]{486}n^\frac{2}{3}< c(n).
\end{align*}

This lower bound comes from a construction where the center points of the balls are positioned on the face-centered cubic lattice such that the convex hull of all the center points forms a regular octahedron with an edge length of $k-1$ and exactly $k$ center points per edge. 
The face-centered cubic lattice, also known as the $A_3$ lattice, is defined as 
\[
A_3=\{n_1(\tfrac{1}{\sqrt{2}},\tfrac{1}{\sqrt{2}},0)+n_2(-\tfrac{1}{\sqrt{2}},\tfrac{1}{\sqrt{2}},0)+n_3(0,\tfrac{1}{\sqrt{2}},\tfrac{1}{\sqrt{2}}):n_1,n_2,n_3 \in \mathbb{Z}\}.
\]
Here we have scaled the lattice from its usual definition so that the shortest non-zero lattice vector has unit length.

The determination of $c(n)$ for small values of $n$ has also received much attention. 
The trivial values of $c(1)=0$, $c(2)=1$, $c(3)=3$, $c(4)=6$, and $c(5)=9$ are known, but surprisingly, for no $n\geq 6$ has there been any rigorous determination of the value of $c(n)$. 
For small values of $n$ there have been attempts to enumerate all rigid packings of $n$ unit-diameter balls, such as \cite{EmpAMB2011} and \cite{EmpHolmes-Cerfon2016}.
However, these enumerations may be incomplete due to rounding errors in the floating-point arithmetic done to generate them.
Additionally, due to the rigidity assumptions these enumerations make, they may not contain all extremal packings of unit-diameter balls.
These approaches do give lower bounds for $c(n)$ for small values of $n$. 
In particular, they give constructions showing that $c(n)\geq 3n-6$ for $n=6,7,8,$ and $9$.
A summary of their approaches and the lower bounds they obtain for $n\geq 10$ can be found in \cite{BezKhan2018}.

The first result in this paper shows that $c(n)=3n-6$ for $n=6,7,8,$ and $9$. In the course of proving this we also enumerate the contact graphs of all extremal packings of $n$ unit-diameter balls for $n=6,7$, and $8$; this is the second result of this paper. These enumerations turn out to match the ones found in \cite{EmpAMB2011}.

\begin{theorem}\label{The: Small Packings}
    $c(n)=3n-6$ for $n=6,7,8,$ and $9$.
\end{theorem}

\newsavebox{\graphA}
\newsavebox{\graphB}
\newsavebox{\graphC}
\newsavebox{\graphD}
\newsavebox{\graphE}
\newsavebox{\graphF}
\newsavebox{\graphG}
\newsavebox{\graphH}
\newsavebox{\graphI}
\newsavebox{\graphJ}
\newsavebox{\graphK}
\newsavebox{\graphL}
\newsavebox{\graphM}
\newsavebox{\graphN}
\newsavebox{\graphO}
\newsavebox{\graphP}
\newsavebox{\graphQ}
\newsavebox{\graphR}
\newsavebox{\graphS}
\newsavebox{\graphT}

\tdplotsetmaincoords{60}{30}
\sbox{\graphA}{\begin{tikzpicture}[tdplot_main_coords, scale=1.416, line cap=round, line join=round]
    \coordinate (a) at (0, 0, 0);
    \coordinate (b) at (1, 0, 0);
    \coordinate (c) at (0.5, 0.8660254038, 0);
    \coordinate (d) at (0.5, 0.2886751346, 0.8164965809);
    \coordinate (e) at (0.5, 0.2886751346, -0.8164965809);
    \coordinate (f) at (0.5, -0.6735753141, -0.544331054);
    \draw[thick] (a) -- (b);
    \draw[thick] (a) -- (c);
    \draw[thick] (b) -- (c);
    \draw[thick] (a) -- (d);
    \draw[thick] (b) -- (d);
    \draw[thick] (c) -- (d);
    \draw[thick] (a) -- (e);
    \draw[thick] (b) -- (e);
    \draw[thick] (c) -- (e);
    \draw[thick] (a) -- (f);
    \draw[thick] (b) -- (f);
    \draw[thick] (e) -- (f);
    \foreach \v in {a,b,c,d,e,f} \filldraw (\v) circle (1pt);
\end{tikzpicture}}
\tdplotsetmaincoords{60}{30}
\sbox{\graphB}{\begin{tikzpicture}[tdplot_main_coords, scale=1.416, line cap=round, line join=round]
    \coordinate (a) at (0.7071067812, 0, 0);
    \coordinate (b) at (-0.7071067812, 0, 0);
    \coordinate (c) at (0, 0.7071067812, 0);
    \coordinate (d) at (0, -0.7071067812, 0);
    \coordinate (e) at (0, 0, 0.7071067812);
    \coordinate (f) at (0, 0, -0.7071067812);
    \draw[thick] (a) -- (c);
    \draw[thick] (a) -- (d);
    \draw[thick] (a) -- (e);
    \draw[thick] (a) -- (f);
    \draw[thick] (b) -- (c);
    \draw[thick] (b) -- (d);
    \draw[thick] (b) -- (e);
    \draw[thick] (b) -- (f);
    \draw[thick] (c) -- (e);
    \draw[thick] (c) -- (f);
    \draw[thick] (d) -- (e);
    \draw[thick] (d) -- (f);
    \foreach \v in {a,b,c,d,e,f} \filldraw (\v) circle (1pt);
\end{tikzpicture}}
\tdplotsetmaincoords{60}{90}
\sbox{\graphC}{\begin{tikzpicture}[tdplot_main_coords, scale=1.416, line cap=round, line join=round]
    \coordinate (a) at (0, 0, 0);
    \coordinate (b) at (1, 0, 0);
    \coordinate (c) at (0.5, 0.8660254038, 0);
    \coordinate (d) at (0.5, 0.2886751346, 0.8164965809);
    \coordinate (e) at (0.5, 0.2886751346, -0.8164965809);
    \coordinate (f) at (0.5, -0.6735753141, -0.544331054);
    \coordinate (g) at (0.5, -0.6735753141, 0.544331054);
    \draw[thick] (a) -- (b);
    \draw[thick] (a) -- (c);
    \draw[thick] (b) -- (c);
    \draw[thick] (a) -- (d);
    \draw[thick] (b) -- (d);
    \draw[thick] (c) -- (d);
    \draw[thick] (a) -- (e);
    \draw[thick] (b) -- (e);
    \draw[thick] (c) -- (e);
    \draw[thick] (a) -- (f);
    \draw[thick] (b) -- (f);
    \draw[thick] (e) -- (f);
    \draw[thick] (a) -- (g);
    \draw[thick] (b) -- (g);
    \draw[thick] (d) -- (g);
    \foreach \v in {a,b,c,d,e,f,g} \filldraw (\v) circle (1pt);
\end{tikzpicture}}
\tdplotsetmaincoords{60}{30}
\sbox{\graphD}{\begin{tikzpicture}[tdplot_main_coords, scale=1.416, line cap=round, line join=round]
    \coordinate (a) at (0, 0, 0);
    \coordinate (b) at (1, 0, 0);
    \coordinate (c) at (0.5, 0.8660254038, 0);
    \coordinate (d) at (0.5, 0.2886751346, 0.8164965809);
    \coordinate (e) at (0.5, 0.2886751346, -0.8164965809);
    \coordinate (f) at (0.5, -0.6735753141, -0.544331054);
    \coordinate (g) at (1.333333333, 0.7698003589, -0.544331054);
    \draw[thick] (a) -- (b);
    \draw[thick] (a) -- (c);
    \draw[thick] (b) -- (c);
    \draw[thick] (a) -- (d);
    \draw[thick] (b) -- (d);
    \draw[thick] (c) -- (d);
    \draw[thick] (a) -- (e);
    \draw[thick] (b) -- (e);
    \draw[thick] (c) -- (e);
    \draw[thick] (a) -- (f);
    \draw[thick] (b) -- (f);
    \draw[thick] (e) -- (f);
    \draw[thick] (b) -- (g);
    \draw[thick] (c) -- (g);
    \draw[thick] (e) -- (g);
    \foreach \v in {a,b,c,d,e,f,g} \filldraw (\v) circle (1pt);
\end{tikzpicture}}
\tdplotsetmaincoords{60}{30}
\sbox{\graphE}{\begin{tikzpicture}[tdplot_main_coords, scale=1.416, line cap=round, line join=round]
    \coordinate (a) at (0, 0, 0);
    \coordinate (b) at (1, 0, 0);
    \coordinate (c) at (0.5, 0.8660254038, 0);
    \coordinate (d) at (0.5, 0.2886751346, 0.8164965809);
    \coordinate (e) at (0.5, 0.2886751346, -0.8164965809);
    \coordinate (f) at (0.5, -0.6735753141, -0.544331054);
    \coordinate (g) at (-0.3333333333, 0.7698003589, 0.544331054);
    \draw[thick] (a) -- (b);
    \draw[thick] (a) -- (c);
    \draw[thick] (b) -- (c);
    \draw[thick] (a) -- (d);
    \draw[thick] (b) -- (d);
    \draw[thick] (c) -- (d);
    \draw[thick] (a) -- (e);
    \draw[thick] (b) -- (e);
    \draw[thick] (c) -- (e);
    \draw[thick] (a) -- (f);
    \draw[thick] (b) -- (f);
    \draw[thick] (e) -- (f);
    \draw[thick] (a) -- (g);
    \draw[thick] (c) -- (g);
    \draw[thick] (d) -- (g);
    \foreach \v in {a,b,c,d,e,f,g} \filldraw (\v) circle (1pt);
\end{tikzpicture}}
\tdplotsetmaincoords{70}{35}
\sbox{\graphF}{\begin{tikzpicture}[tdplot_main_coords, scale=1.416, line cap=round, line join=round]
    \coordinate (a) at (0.8506508084, 0, 0);
    \coordinate (b) at (0.2628655561, 0.8090169944, 0);
    \coordinate (c) at (-0.6881909602, 0.5, 0);
    \coordinate (d) at (-0.6881909602, -0.5, 0);
    \coordinate (e) at (0.2628655561, -0.8090169944, 0);
    \coordinate (f) at (0, 0, 0.5257311121);
    \coordinate (g) at (0, 0, -0.5257311121);
    \draw[thick] (a) -- (b);
    \draw[thick] (b) -- (c);
    \draw[thick] (c) -- (d);
    \draw[thick] (d) -- (e);
    \draw[thick] (e) -- (a);
    \draw[thick] (a) -- (f);
    \draw[thick] (b) -- (f);
    \draw[thick] (c) -- (f);
    \draw[thick] (d) -- (f);
    \draw[thick] (e) -- (f);
    \draw[thick] (a) -- (g);
    \draw[thick] (b) -- (g);
    \draw[thick] (c) -- (g);
    \draw[thick] (d) -- (g);
    \draw[thick] (e) -- (g);
    \foreach \v in {a,b,c,d,e,f,g} \filldraw (\v) circle (1pt);
\end{tikzpicture}}
\tdplotsetmaincoords{60}{30}
\sbox{\graphG}{\begin{tikzpicture}[tdplot_main_coords, scale=1.416, line cap=round, line join=round]
    \coordinate (a) at (0.7071067812, 0, 0);
    \coordinate (b) at (-0.7071067812, 0, 0);
    \coordinate (c) at (0, 0.7071067812, 0);
    \coordinate (d) at (0, -0.7071067812, 0);
    \coordinate (e) at (0, 0, 0.7071067812);
    \coordinate (f) at (0, 0, -0.7071067812);
    \coordinate (g) at (0.7071067812, 0.7071067812, 0.7071067812);
    \draw[thick] (a) -- (c);
    \draw[thick] (a) -- (d);
    \draw[thick] (a) -- (e);
    \draw[thick] (a) -- (f);
    \draw[thick] (b) -- (c);
    \draw[thick] (b) -- (d);
    \draw[thick] (b) -- (e);
    \draw[thick] (b) -- (f);
    \draw[thick] (c) -- (e);
    \draw[thick] (c) -- (f);
    \draw[thick] (d) -- (e);
    \draw[thick] (d) -- (f);
    \draw[thick] (a) -- (g);
    \draw[thick] (c) -- (g);
    \draw[thick] (e) -- (g);
    \foreach \v in {a,b,c,d,e,f,g} \filldraw (\v) circle (1pt);
\end{tikzpicture}}
\tdplotsetmaincoords{60}{54}
\sbox{\graphH}{\begin{tikzpicture}[tdplot_main_coords, scale=1.416, line cap=round, line join=round]
    \coordinate (a) at (0, 0, 0);
    \coordinate (b) at (1, 0, 0);
    \coordinate (c) at (0.5, 0.8660254038, 0);
    \coordinate (d) at (0.5, 0.2886751346, 0.8164965809);
    \coordinate (e) at (0.5, 0.2886751346, -0.8164965809);
    \coordinate (f) at (0.5, -0.6735753141, -0.544331054);
    \coordinate (g) at (1.333333333, 0.7698003589, -0.544331054);
    \coordinate (h) at (-0.3333333333, 0.7698003589, -0.544331054);
    \draw[thick] (a) -- (b);
    \draw[thick] (a) -- (c);
    \draw[thick] (b) -- (c);
    \draw[thick] (a) -- (d);
    \draw[thick] (b) -- (d);
    \draw[thick] (c) -- (d);
    \draw[thick] (a) -- (e);
    \draw[thick] (b) -- (e);
    \draw[thick] (c) -- (e);
    \draw[thick] (a) -- (f);
    \draw[thick] (b) -- (f);
    \draw[thick] (e) -- (f);
    \draw[thick] (b) -- (g);
    \draw[thick] (c) -- (g);
    \draw[thick] (e) -- (g);
    \draw[thick] (a) -- (h);
    \draw[thick] (c) -- (h);
    \draw[thick] (e) -- (h);
    \foreach \v in {a,b,c,d,e,f,g,h} \filldraw (\v) circle (1pt);
\end{tikzpicture}}
\tdplotsetmaincoords{70}{50}
\sbox{\graphI}{\begin{tikzpicture}[tdplot_main_coords, scale=1.416, line cap=round, line join=round]
    \coordinate (a) at (0, 0, 0);
    \coordinate (b) at (1, 0, 0);
    \coordinate (c) at (0.5, 0.8660254038, 0);
    \coordinate (d) at (0.5, 0.2886751346, 0.8164965809);
    \coordinate (e) at (0.5, 0.2886751346, -0.8164965809);
    \coordinate (f) at (0.5, -0.6735753141, -0.544331054);
    \coordinate (g) at (0.5, -0.6735753141, 0.544331054);
    \coordinate (h) at (-0.3333333333, -0.2566001196, 0.9072184233);
    \draw[thick] (a) -- (b);
    \draw[thick] (a) -- (c);
    \draw[thick] (b) -- (c);
    \draw[thick] (a) -- (d);
    \draw[thick] (b) -- (d);
    \draw[thick] (c) -- (d);
    \draw[thick] (a) -- (e);
    \draw[thick] (b) -- (e);
    \draw[thick] (c) -- (e);
    \draw[thick] (a) -- (f);
    \draw[thick] (b) -- (f);
    \draw[thick] (e) -- (f);
    \draw[thick] (a) -- (g);
    \draw[thick] (b) -- (g);
    \draw[thick] (d) -- (g);
    \draw[thick] (a) -- (h);
    \draw[thick] (d) -- (h);
    \draw[thick] (g) -- (h);
    \foreach \v in {a,b,c,d,e,f,g,h} \filldraw (\v) circle (1pt);
\end{tikzpicture}}
\tdplotsetmaincoords{60}{30}
\sbox{\graphJ}{\begin{tikzpicture}[tdplot_main_coords, scale=1.416, line cap=round, line join=round]
    \coordinate (a) at (0, 0, 0);
    \coordinate (b) at (1, 0, 0);
    \coordinate (c) at (0.5, 0.8660254038, 0);
    \coordinate (d) at (0.5, 0.2886751346, 0.8164965809);
    \coordinate (e) at (0.5, 0.2886751346, -0.8164965809);
    \coordinate (f) at (0.5, -0.6735753141, -0.544331054);
    \coordinate (g) at (1.333333333, 0.7698003589, -0.544331054);
    \coordinate (h) at (1.333333333, 0.7698003589, 0.544331054);
    \draw[thick] (a) -- (b);
    \draw[thick] (a) -- (c);
    \draw[thick] (b) -- (c);
    \draw[thick] (a) -- (d);
    \draw[thick] (b) -- (d);
    \draw[thick] (c) -- (d);
    \draw[thick] (a) -- (e);
    \draw[thick] (b) -- (e);
    \draw[thick] (c) -- (e);
    \draw[thick] (a) -- (f);
    \draw[thick] (b) -- (f);
    \draw[thick] (e) -- (f);
    \draw[thick] (b) -- (g);
    \draw[thick] (c) -- (g);
    \draw[thick] (e) -- (g);
    \draw[thick] (b) -- (h);
    \draw[thick] (c) -- (h);
    \draw[thick] (d) -- (h);
    \foreach \v in {a,b,c,d,e,f,g,h} \filldraw (\v) circle (1pt);
\end{tikzpicture}}
\tdplotsetmaincoords{60}{30}
\sbox{\graphK}{\begin{tikzpicture}[tdplot_main_coords, scale=1.416, line cap=round, line join=round]
    \coordinate (a) at (0, 0, 0);
    \coordinate (b) at (1, 0, 0);
    \coordinate (c) at (0.5, 0.8660254038, 0);
    \coordinate (d) at (0.5, 0.2886751346, 0.8164965809);
    \coordinate (e) at (0.5, 0.2886751346, -0.8164965809);
    \coordinate (f) at (0.5, -0.6735753141, -0.544331054);
    \coordinate (g) at (1.333333333, 0.7698003589, -0.544331054);
    \coordinate (h) at (-0.3333333333, 0.7698003589, 0.544331054);
    \draw[thick] (a) -- (b);
    \draw[thick] (a) -- (c);
    \draw[thick] (b) -- (c);
    \draw[thick] (a) -- (d);
    \draw[thick] (b) -- (d);
    \draw[thick] (c) -- (d);
    \draw[thick] (a) -- (e);
    \draw[thick] (b) -- (e);
    \draw[thick] (c) -- (e);
    \draw[thick] (a) -- (f);
    \draw[thick] (b) -- (f);
    \draw[thick] (e) -- (f);
    \draw[thick] (b) -- (g);
    \draw[thick] (c) -- (g);
    \draw[thick] (e) -- (g);
    \draw[thick] (a) -- (h);
    \draw[thick] (c) -- (h);
    \draw[thick] (d) -- (h);
    \foreach \v in {a,b,c,d,e,f,g,h} \filldraw (\v) circle (1pt);
\end{tikzpicture}}
\tdplotsetmaincoords{60}{30}
\sbox{\graphL}{\begin{tikzpicture}[tdplot_main_coords, scale=1.416, line cap=round, line join=round]
    \coordinate (a) at (0.7071067812, 0, 0);
    \coordinate (b) at (-0.7071067812, 0, 0);
    \coordinate (c) at (0, 0.7071067812, 0);
    \coordinate (d) at (0, -0.7071067812, 0);
    \coordinate (e) at (0, 0, 0.7071067812);
    \coordinate (f) at (0, 0, -0.7071067812);
    \coordinate (g) at (0.7071067812, 0.7071067812, 0.7071067812);
    \coordinate (h) at (0.7071067812, 0.7071067812, -0.7071067812);
    \draw[thick] (a) -- (c);
    \draw[thick] (a) -- (d);
    \draw[thick] (a) -- (e);
    \draw[thick] (a) -- (f);
    \draw[thick] (b) -- (c);
    \draw[thick] (b) -- (d);
    \draw[thick] (b) -- (e);
    \draw[thick] (b) -- (f);
    \draw[thick] (c) -- (e);
    \draw[thick] (c) -- (f);
    \draw[thick] (d) -- (e);
    \draw[thick] (d) -- (f);
    \draw[thick] (a) -- (g);
    \draw[thick] (c) -- (g);
    \draw[thick] (e) -- (g);
    \draw[thick] (a) -- (h);
    \draw[thick] (c) -- (h);
    \draw[thick] (f) -- (h);
    \foreach \v in {a,b,c,d,e,f,g,h} \filldraw (\v) circle (1pt);
\end{tikzpicture}}
\tdplotsetmaincoords{40}{90}
\sbox{\graphM}{\begin{tikzpicture}[tdplot_main_coords, scale=1.416, line cap=round, line join=round]
    \coordinate (a) at (0, 0, 0);
    \coordinate (b) at (1, 0, 0);
    \coordinate (c) at (0.5, 0.8660254038, 0);
    \coordinate (d) at (0.5, 0.2886751346, 0.8164965809);
    \coordinate (e) at (0.5, 0.2886751346, -0.8164965809);
    \coordinate (f) at (0.5, -0.6735753141, -0.544331054);
    \coordinate (g) at (0.5, -0.6735753141, 0.544331054);
    \coordinate (h) at (1.289473684, -0.9571859726, 0);
    \draw[thick] (a) -- (b);
    \draw[thick] (a) -- (c);
    \draw[thick] (b) -- (c);
    \draw[thick] (a) -- (d);
    \draw[thick] (b) -- (d);
    \draw[thick] (c) -- (d);
    \draw[thick] (a) -- (e);
    \draw[thick] (b) -- (e);
    \draw[thick] (c) -- (e);
    \draw[thick] (a) -- (f);
    \draw[thick] (b) -- (f);
    \draw[thick] (e) -- (f);
    \draw[thick] (a) -- (g);
    \draw[thick] (b) -- (g);
    \draw[thick] (d) -- (g);
    \draw[thick] (b) -- (h);
    \draw[thick] (f) -- (h);
    \draw[thick] (g) -- (h);
    \foreach \v in {a,b,c,d,e,f,g,h} \filldraw (\v) circle (1pt);
\end{tikzpicture}}
\tdplotsetmaincoords{60}{30}
\sbox{\graphN}{\begin{tikzpicture}[tdplot_main_coords, scale=1.416, line cap=round, line join=round]
    \coordinate (a) at (0.8506508084, 0, 0);
    \coordinate (b) at (0.2628655561, 0.8090169944, 0);
    \coordinate (c) at (-0.6881909602, 0.5, 0);
    \coordinate (d) at (-0.6881909602, -0.5, 0);
    \coordinate (e) at (0.2628655561, -0.8090169944, 0);
    \coordinate (f) at (0, 0, 0.5257311121);
    \coordinate (g) at (0, 0, -0.5257311121);
    \coordinate (h) at (0.7721727581, 0.5610163478, 0.8240763637);
    \draw[thick] (a) -- (b);
    \draw[thick] (b) -- (c);
    \draw[thick] (c) -- (d);
    \draw[thick] (d) -- (e);
    \draw[thick] (e) -- (a);
    \draw[thick] (a) -- (f);
    \draw[thick] (b) -- (f);
    \draw[thick] (c) -- (f);
    \draw[thick] (d) -- (f);
    \draw[thick] (e) -- (f);
    \draw[thick] (a) -- (g);
    \draw[thick] (b) -- (g);
    \draw[thick] (c) -- (g);
    \draw[thick] (d) -- (g);
    \draw[thick] (e) -- (g);
    \draw[thick] (a) -- (h);
    \draw[thick] (b) -- (h);
    \draw[thick] (f) -- (h);
    \foreach \v in {a,b,c,d,e,f,g,h} \filldraw (\v) circle (1pt);
\end{tikzpicture}}
\tdplotsetmaincoords{60}{30}
\sbox{\graphO}{\begin{tikzpicture}[tdplot_main_coords, scale=1.416, line cap=round, line join=round]
    \coordinate (a) at (0.7071067812, 0, 0);
    \coordinate (b) at (-0.7071067812, 0, 0);
    \coordinate (c) at (0, 0.7071067812, 0);
    \coordinate (d) at (0, -0.7071067812, 0);
    \coordinate (e) at (0, 0, 0.7071067812);
    \coordinate (f) at (0, 0, -0.7071067812);
    \coordinate (g) at (0.7071067812, 0.7071067812, 0.7071067812);
    \coordinate (h) at (-0.7071067812, -0.7071067812, 0.7071067812);
    \draw[thick] (a) -- (c);
    \draw[thick] (a) -- (d);
    \draw[thick] (a) -- (e);
    \draw[thick] (a) -- (f);
    \draw[thick] (b) -- (c);
    \draw[thick] (b) -- (d);
    \draw[thick] (b) -- (e);
    \draw[thick] (b) -- (f);
    \draw[thick] (c) -- (e);
    \draw[thick] (c) -- (f);
    \draw[thick] (d) -- (e);
    \draw[thick] (d) -- (f);
    \draw[thick] (a) -- (g);
    \draw[thick] (c) -- (g);
    \draw[thick] (e) -- (g);
    \draw[thick] (b) -- (h);
    \draw[thick] (d) -- (h);
    \draw[thick] (e) -- (h);
    \foreach \v in {a,b,c,d,e,f,g,h} \filldraw (\v) circle (1pt);
\end{tikzpicture}}
\tdplotsetmaincoords{60}{30}
\sbox{\graphP}{\begin{tikzpicture}[tdplot_main_coords, scale=1.416, line cap=round, line join=round]
    \coordinate (a) at (0.7071067812, 0, 0);
    \coordinate (b) at (-0.7071067812, 0, 0);
    \coordinate (c) at (0, 0.7071067812, 0);
    \coordinate (d) at (0, -0.7071067812, 0);
    \coordinate (e) at (0, 0, 0.7071067812);
    \coordinate (f) at (0, 0, -0.7071067812);
    \coordinate (g) at (0.7071067812, 0.7071067812, 0.7071067812);
    \coordinate (h) at (0.9428090416, 0.9428090416, -0.2357022604);
    \draw[thick] (a) -- (c);
    \draw[thick] (a) -- (d);
    \draw[thick] (a) -- (e);
    \draw[thick] (a) -- (f);
    \draw[thick] (b) -- (c);
    \draw[thick] (b) -- (d);
    \draw[thick] (b) -- (e);
    \draw[thick] (b) -- (f);
    \draw[thick] (c) -- (e);
    \draw[thick] (c) -- (f);
    \draw[thick] (d) -- (e);
    \draw[thick] (d) -- (f);
    \draw[thick] (a) -- (g);
    \draw[thick] (c) -- (g);
    \draw[thick] (e) -- (g);
    \draw[thick] (a) -- (h);
    \draw[thick] (c) -- (h);
    \draw[thick] (g) -- (h);
    \foreach \v in {a,b,c,d,e,f,g,h} \filldraw (\v) circle (1pt);
\end{tikzpicture}}
\tdplotsetmaincoords{60}{30}
\sbox{\graphQ}{\begin{tikzpicture}[tdplot_main_coords, scale=1.416, line cap=round, line join=round]
    \coordinate (a) at (0, 0, 0);
    \coordinate (b) at (1, 0, 0);
    \coordinate (c) at (0.5, 0.8660254038, 0);
    \coordinate (d) at (0.5, 0.2886751346, 0.8164965809);
    \coordinate (e) at (0.5, 0.2886751346, -0.8164965809);
    \coordinate (f) at (0.5, -0.6735753141, -0.544331054);
    \coordinate (g) at (-0.3333333333, 0.7698003589, 0.544331054);
    \coordinate (h) at (-0.3888888889, -0.1603750748, 0.9072184233);
    \draw[thick] (a) -- (b);
    \draw[thick] (a) -- (c);
    \draw[thick] (b) -- (c);
    \draw[thick] (a) -- (d);
    \draw[thick] (b) -- (d);
    \draw[thick] (c) -- (d);
    \draw[thick] (a) -- (e);
    \draw[thick] (b) -- (e);
    \draw[thick] (c) -- (e);
    \draw[thick] (a) -- (f);
    \draw[thick] (b) -- (f);
    \draw[thick] (e) -- (f);
    \draw[thick] (a) -- (g);
    \draw[thick] (c) -- (g);
    \draw[thick] (d) -- (g);
    \draw[thick] (a) -- (h);
    \draw[thick] (d) -- (h);
    \draw[thick] (g) -- (h);
    \foreach \v in {a,b,c,d,e,f,g,h} \filldraw (\v) circle (1pt);
\end{tikzpicture}}
\tdplotsetmaincoords{60}{30}
\sbox{\graphR}{\begin{tikzpicture}[tdplot_main_coords, scale=1.416, line cap=round, line join=round]
    \coordinate (a) at (0, 0, 0);
    \coordinate (b) at (1, 0, 0);
    \coordinate (c) at (0.5, 0.8660254038, 0);
    \coordinate (d) at (0.5, 0.2886751346, 0.8164965809);
    \coordinate (e) at (0.5, 0.2886751346, -0.8164965809);
    \coordinate (f) at (0.5, -0.6735753141, -0.544331054);
    \coordinate (g) at (-0.3333333333, 0.7698003589, 0.544331054);
    \coordinate (h) at (1.333333333, -0.2566001196, -0.9072184233);
    \draw[thick] (a) -- (b);
    \draw[thick] (a) -- (c);
    \draw[thick] (b) -- (c);
    \draw[thick] (a) -- (d);
    \draw[thick] (b) -- (d);
    \draw[thick] (c) -- (d);
    \draw[thick] (a) -- (e);
    \draw[thick] (b) -- (e);
    \draw[thick] (c) -- (e);
    \draw[thick] (a) -- (f);
    \draw[thick] (b) -- (f);
    \draw[thick] (e) -- (f);
    \draw[thick] (a) -- (g);
    \draw[thick] (c) -- (g);
    \draw[thick] (d) -- (g);
    \draw[thick] (b) -- (h);
    \draw[thick] (e) -- (h);
    \draw[thick] (f) -- (h);
    \foreach \v in {a,b,c,d,e,f,g,h} \filldraw (\v) circle (1pt);
\end{tikzpicture}}
\tdplotsetmaincoords{60}{30}
\sbox{\graphS}{\begin{tikzpicture}[tdplot_main_coords, scale=1.416, line cap=round, line join=round]
    \coordinate (a) at (0.7071067812, 0, 0);
    \coordinate (b) at (-0.7071067812, 0, 0);
    \coordinate (c) at (0, 0.7071067812, 0);
    \coordinate (d) at (0, -0.7071067812, 0);
    \coordinate (e) at (0, 0, 0.7071067812);
    \coordinate (f) at (0, 0, -0.7071067812);
    \coordinate (g) at (0.7071067812, 0.7071067812, 0.7071067812);
    \coordinate (h) at (-0.7071067812, -0.7071067812, -0.7071067812);
    \draw[thick] (a) -- (c);
    \draw[thick] (a) -- (d);
    \draw[thick] (a) -- (e);
    \draw[thick] (a) -- (f);
    \draw[thick] (b) -- (c);
    \draw[thick] (b) -- (d);
    \draw[thick] (b) -- (e);
    \draw[thick] (b) -- (f);
    \draw[thick] (c) -- (e);
    \draw[thick] (c) -- (f);
    \draw[thick] (d) -- (e);
    \draw[thick] (d) -- (f);
    \draw[thick] (a) -- (g);
    \draw[thick] (c) -- (g);
    \draw[thick] (e) -- (g);
    \draw[thick] (b) -- (h);
    \draw[thick] (d) -- (h);
    \draw[thick] (f) -- (h);
    \foreach \v in {a,b,c,d,e,f,g,h} \filldraw (\v) circle (1pt);
\end{tikzpicture}}
\tdplotsetmaincoords{80}{80}
\sbox{\graphT}{\begin{tikzpicture}[tdplot_main_coords, scale=1.416, line cap=round, line join=round]
    \coordinate (a) at (0, 0, 0);
    \coordinate (b) at (1.719939178, 0, 0);
    \coordinate (c) at (0.2907079543, 0.9568118338, 0);
    \coordinate (d) at (1.429231224, 0.08832574152, -0.9527263241);
    \coordinate (e) at (0.859969589, 0.2612843938, 0.438386555);
    \coordinate (f) at (0.4851980381, 0.3751511616, -0.7898382555);
    \coordinate (g) at (0.859969589, -0.4123948506, -0.3006373118);
    \coordinate (h) at (1.23474114, 0.8210968704, -0.3006373118);
    \draw[thick] (a) -- (c);
    \draw[thick] (a) -- (e);
    \draw[thick] (a) -- (f);
    \draw[thick] (a) -- (g);
    \draw[thick] (b) -- (d);
    \draw[thick] (b) -- (e);
    \draw[thick] (b) -- (g);
    \draw[thick] (b) -- (h);
    \draw[thick] (c) -- (e);
    \draw[thick] (c) -- (f);
    \draw[thick] (c) -- (h);
    \draw[thick] (d) -- (f);
    \draw[thick] (d) -- (g);
    \draw[thick] (d) -- (h);
    \draw[thick] (e) -- (g);
    \draw[thick] (e) -- (h);
    \draw[thick] (f) -- (g);
    \draw[thick] (f) -- (h);
    \foreach \v in {a,b,c,d,e,f,g,h} \filldraw (\v) circle (1pt);
\end{tikzpicture}}

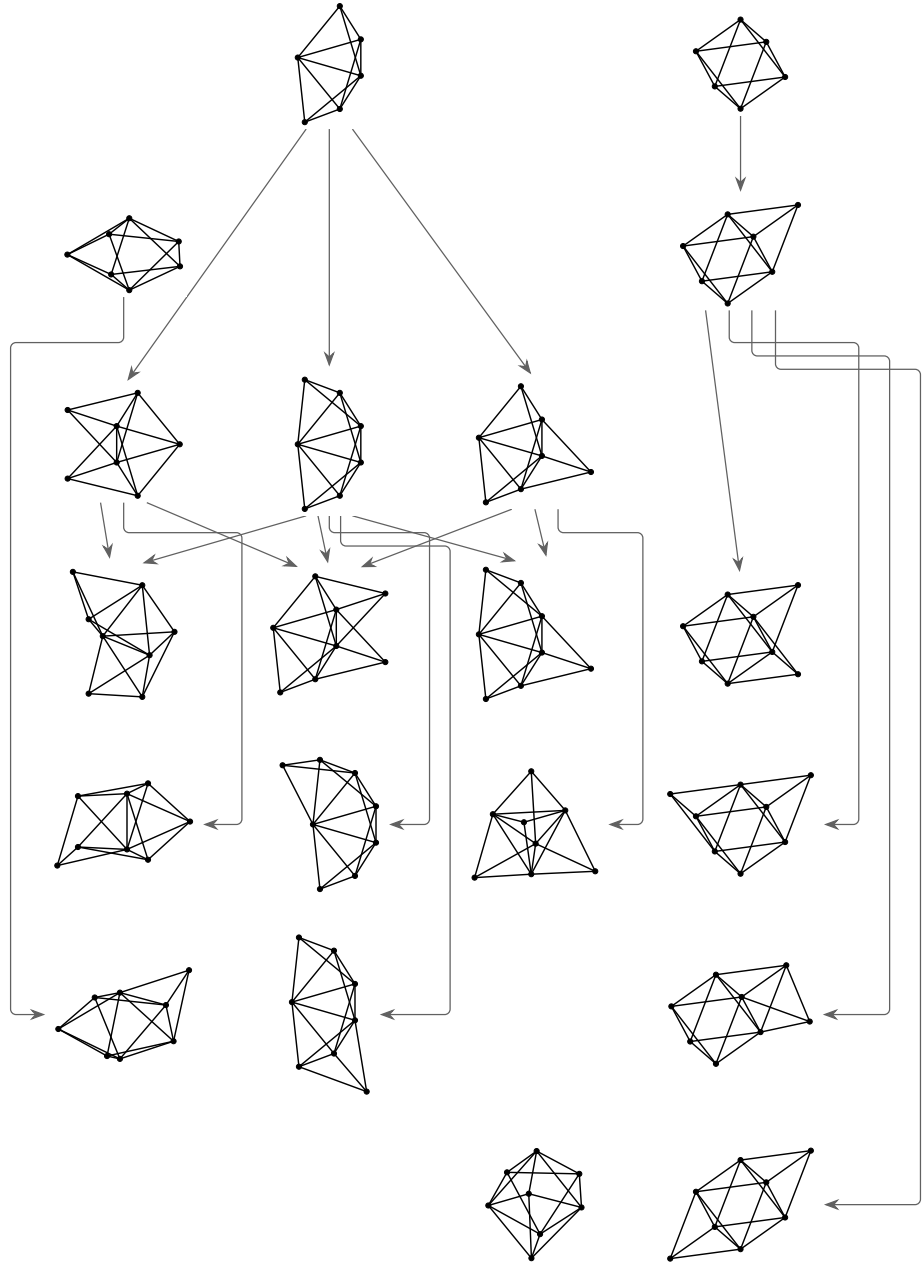
\begin{figure}[p]
  \centering
  \resizebox{\textwidth}{!}{
  \begin{tikzpicture}[
      box/.style={inner sep=2pt, fill=white},
      grow arrow/.style={-{Stealth[length=3mm,width=2.2mm]},
                         line width=0.7pt, black!62, rounded corners=3pt, shorten >=1.2mm}]
    \node[box] (G1) at (4.00, 0.00) {\usebox{\graphA}};
    \node[box] (G2) at (12.00, 0.00) {\usebox{\graphB}};
    \node[box] (G3) at (0.00, -7.40) {\usebox{\graphC}};
    \node[box] (G4) at (8.00, -7.40) {\usebox{\graphD}};
    \node[box] (G5) at (4.00, -7.40) {\usebox{\graphE}};
    \node[box] (G6) at (0.00, -3.70) {\usebox{\graphF}};
    \node[box] (G7) at (12.00, -3.70) {\usebox{\graphG}};
    \node[box] (G8) at (8.00, -14.80) {\usebox{\graphH}};
    \node[box] (G9) at (0.00, -11.10) {\usebox{\graphI}};
    \node[box] (G10) at (4.00, -11.10) {\usebox{\graphJ}};
    \node[box] (G11) at (8.00, -11.10) {\usebox{\graphK}};
    \node[box] (G12) at (12.00, -11.10) {\usebox{\graphL}};
    \node[box] (G13) at (0.00, -14.80) {\usebox{\graphM}};
    \node[box] (G14) at (0.00, -18.50) {\usebox{\graphN}};
    \node[box] (G15) at (12.00, -14.80) {\usebox{\graphO}};
    \node[box] (G16) at (12.00, -18.50) {\usebox{\graphP}};
    \node[box] (G17) at (4.00, -14.80) {\usebox{\graphQ}};
    \node[box] (G18) at (4.00, -18.50) {\usebox{\graphR}};
    \node[box] (G19) at (12.00, -22.20) {\usebox{\graphS}};
    \node[box] (G20) at (8.00, -22.20) {\usebox{\graphT}};
    \begin{pgfonlayer}{behind}
      \draw[grow arrow] ([xshift=-4.5mm]G1.south) -- ([xshift=0.0mm]G3.north);
      \draw[grow arrow] ([xshift=0.0mm]G1.south) -- ([xshift=0.0mm]G5.north);
      \draw[grow arrow] ([xshift=4.5mm]G1.south) -- ([xshift=0.0mm]G4.north);
      \draw[grow arrow] ([xshift=0.0mm]G2.south) -- ([xshift=0.0mm]G7.north);
      \draw[grow arrow] ([xshift=-6.8mm]G7.south) -- ([xshift=0.0mm]G12.north);
      \draw[grow arrow] ([xshift=-4.5mm]G3.south) -- ([xshift=-2.5mm]G9.north);
      \draw[grow arrow] ([xshift=-4.5mm]G5.south) -- ([xshift=2.5mm]G9.north);
      \draw[grow arrow] ([xshift=4.5mm]G3.south) -- ([xshift=-5.0mm]G10.north);
      \draw[grow arrow] ([xshift=-2.2mm]G5.south) -- ([xshift=0.0mm]G10.north);
      \draw[grow arrow] ([xshift=-4.5mm]G4.south) -- ([xshift=5.0mm]G10.north);
      \draw[grow arrow] ([xshift=4.5mm]G5.south) -- ([xshift=-2.5mm]G11.north);
      \draw[grow arrow] ([xshift=0.0mm]G4.south) -- ([xshift=2.5mm]G11.north);
      \draw[grow arrow] ([xshift=0.0mm]G6.south) -- (0.00,-5.42) -- (-2.20,-5.42) -- (-2.20,-18.50) -- (G14.west);
      \draw[grow arrow] ([xshift=0.0mm]G3.south) -- (0.00,-9.12) -- (2.30,-9.12) -- (2.30,-14.80) -- (G13.east);
      \draw[grow arrow] ([xshift=0.0mm]G5.south) -- (4.00,-9.12) -- (5.95,-9.12) -- (5.95,-14.80) -- (G17.east);
      \draw[grow arrow] ([xshift=4.5mm]G4.south) -- (8.45,-9.12) -- (10.10,-9.12) -- (10.10,-14.80) -- (G8.east);
      \draw[grow arrow] ([xshift=-2.2mm]G7.south) -- (11.78,-5.42) -- (14.30,-5.42) -- (14.30,-14.80) -- (G15.east);
      \draw[grow arrow] ([xshift=2.2mm]G5.south) -- (4.22,-9.38) -- (6.35,-9.38) -- (6.35,-18.50) -- (G18.east);
      \draw[grow arrow] ([xshift=2.2mm]G7.south) -- (12.22,-5.68) -- (14.90,-5.68) -- (14.90,-18.50) -- (G16.east);
      \draw[grow arrow] ([xshift=6.8mm]G7.south) -- (12.68,-5.94) -- (15.50,-5.94) -- (15.50,-22.20) -- (G19.east);
    \end{pgfonlayer}
  \end{tikzpicture}}
  \caption{All extremal $n$-vertex contact graphs for $n=6$ (two graphs in the top row), $n=7$ (five graphs in rows two and three), and $n=8$ (thirteen graphs in the bottom four rows); an arrow from graph $G$ to $G'$ indicates $G$ can be obtained from $G'$ by deleting a vertex}
  \label{Fig: All}
\end{figure}

\begin{theorem}\label{The: Enumeration}
    Figure \ref{Fig: All} depicts all extremal contact graphs of $n$ unit-diameter balls for $n=6,7,$ and $8$.
    
\end{theorem}

A packing of $n$ unit-diameter balls is said to be \textit{minimally rigid} \cite{EmpAMB2011} if it has a contact number of at least $3n-6$ and every ball is in contact with at least three others. The third result of this paper shows extremal packings are minimally rigid. This, together with Theorem \ref{The: Small Packings}, resolves a conjecture of K. Bezdek and Khan found in \cite{BezKhan2018}.

\begin{theorem}\label{The: Rigid}
    For $n\geq 4$, every extremal packing of $n$ unit-diameter balls in $\mathbb{R}^3$ is minimally rigid. 
\end{theorem}

K. Bezdek in \cite{Bez} considered a restricted problem where the center point of each ball in the packing lies on some lattice $\Lambda\subseteq\mathbb{R}^3$ in which the shortest non-zero lattice vector has a length of $1$.
For each such lattice $\Lambda$, let $c_{\Lambda}(n)$ be the maximum contact number a packing of $n$ unit-diameter balls can have where the center point of each ball in the packing lies on $\Lambda$. 
Of course one has that $c_\Lambda(n)\leq c(n)$ for all $n$ and any lattice $\Lambda$.
K. Bezdek showed in \cite{Bez} that $c_\Lambda(n)\leq c_{A_{3}}(n)$ for all lattices $\Lambda$ and all $n$.
So, the face-centered cubic lattice is the extremal lattice to consider when trying to maximize the contact number when each ball is centered on a lattice. 
To simplify notation, let $c_A(n):= c_{A_3}(n)$ for each $n$.
Additionally, he showed the following upper bound for all $n\geq 2$:
\begin{align*}
 c_{A}(n)\leq 6n- \frac{3\sqrt[3]{18\uppi}}{\uppi}n^\frac{2}{3}=6n-3.665\dots n^\frac{2}{3}.
\end{align*}

While some intuition on why $c_{A}(n)$ is a good candidate for $c(n)$ has been stated before in \cite{BezKhan2018}, to our knowledge no one has explicitly put forward a conjecture. 

\begin{conjecture}\label{Conj: FCC is the best}
    For all $n$ sufficiently large, $c_{A}(n)=c(n)$.
\end{conjecture}

The fourth result of this paper is an improved upper bound on $c_{A}(n)$ for all $n$. 
To achieve this bound we utilize the geometric Brascamp-Lieb inequality.

\begin{theorem}\label{The: FCC Upper Bound}
    $c_{A}(n)\leq 6n-\frac{6}{\sqrt[6]{2}}n^\frac{2}{3}=6n-5.345\dots n^\frac{2}{3}$ for all $n\in \mathbb{N}$.
\end{theorem}

The fifth result is the determination of the asymptotics of $c_{A}(n)$, which, if Conjecture \ref{Conj: FCC is the best} is true, would determine the asymptotics of $c(n)$. 
We obtain this by transforming the contact number problem on the face-centered cubic lattice into an edge isoperimetric problem on a particular Cayley graph.
We then apply a discrete isoperimetric result of Barber and Erde found in \cite{BarErd2018}.

\begin{theorem}\label{The: FCC asymptotics}
    $c_{A}(n)=6n-(1+o(1))6\sqrt[3]{2}n^\frac{2}{3} =6n-(1+o(1))7.559\dots n^\frac{2}{3}$.
\end{theorem}

The sixth result is an improved lower bound on $c_{A}(n)$ for infinitely many values of $n$. 
The results in \cite{BarErd2018} and \cite{BarErdeKevRob2023}, when applied to this problem, state that asymptotically the approximate shape of the center points in an extremal packing of balls centered on $A_3$ will be the intersection of the $A_3$ lattice with an appropriately scaled truncated octahedron (an Archimedean solid depicted in Figure \ref{Fig: TruOct}).
This lower bound comes from calculating explicitly the contact number such a construction has when the truncated octahedron with an edge length of $k-1$ with exactly $k$ lattice points on each edge is intersected with the face-centered cubic lattice.

\begin{figure}
    \centering
    \begin{subfigure}[t]{0.45\textwidth}
\tdplotsetmaincoords{70}{135}

\begin{tikzpicture}[tdplot_main_coords, scale=1.2, line join=round]

    \coordinate (V1)  at (0, 1, 2);   \coordinate (V2)  at (0, -1, 2);
    \coordinate (V3)  at (1, 0, 2);   \coordinate (V4)  at (-1, 0, 2);
    \coordinate (V5)  at (0, 2, 1);   \coordinate (V6)  at (0, -2, 1);
    \coordinate (V7)  at (2, 0, 1);   \coordinate (V8)  at (-2, 0, 1);
    \coordinate (V9)  at (1, 2, 0);   \coordinate (V10) at (-1, 2, 0);
    \coordinate (V11) at (1, -2, 0);  \coordinate (V12) at (-1, -2, 0);
    \coordinate (V13) at (2, 1, 0);   \coordinate (V14) at (-2, 1, 0);
    \coordinate (V15) at (2, -1, 0);  \coordinate (V16) at (-2, -1, 0);
    \coordinate (V17) at (0, 2, -1);  \coordinate (V18) at (0, -2, -1);
    \coordinate (V19) at (2, 0, -1);  \coordinate (V20) at (-2, 0, -1);
    \coordinate (V21) at (0, 1, -2);  \coordinate (V22) at (0, -1, -2);
    \coordinate (V23) at (1, 0, -2);  \coordinate (V24) at (-1, 0, -2);

    \begin{scope}[dashed, very thick, fill opacity=0.3]
        \filldraw[fill=yellow!40] (V21)--(V23)--(V22)--(V24)--cycle; 
        \filldraw[fill=yellow!40] (V6)--(V11)--(V18)--(V12)--cycle;
        \filldraw[fill=yellow!40] (V8)--(V14)--(V20)--(V16)--cycle;
        \filldraw[fill=blue!30] (V22)--(V23)--(V19)--(V15)--(V11)--(V18)--cycle;
        \filldraw[fill=blue!30] (V22)--(V18)--(V12)--(V16)--(V20)--(V24)--cycle;
        \filldraw[fill=blue!30] (V21)--(V17)--(V10)--(V14)--(V20)--(V24)--cycle;
        \filldraw[fill=blue!30] (V2)--(V4)--(V8)--(V16)--(V12)--(V6)--cycle;
    \end{scope}

    \begin{scope}[very thick, fill opacity=0.7]
        \filldraw[fill=yellow!60] (V1)--(V3)--(V2)--(V4)--cycle;     
        \filldraw[fill=yellow!60] (V5)--(V9)--(V17)--(V10)--cycle;   
        \filldraw[fill=yellow!60] (V7)--(V13)--(V19)--(V15)--cycle;  
        \filldraw[fill=blue!50] (V1)--(V3)--(V7)--(V13)--(V9)--(V5)--cycle;
        \filldraw[fill=blue!50] (V1)--(V5)--(V10)--(V14)--(V8)--(V4)--cycle;
        \filldraw[fill=blue!50] (V2)--(V3)--(V7)--(V15)--(V11)--(V6)--cycle;
        \filldraw[fill=blue!50] (V21)--(V23)--(V19)--(V13)--(V9)--(V17)--cycle;
    \end{scope}

\end{tikzpicture}
    \end{subfigure}
    \hfill
    \begin{subfigure}[t]{0.45\textwidth}
\tdplotsetmaincoords{90}{135}

\begin{tikzpicture}[tdplot_main_coords, scale=1.2, line join=round]

    \coordinate (V1)  at (0, 1, 2);   \coordinate (V2)  at (0, -1, 2);
    \coordinate (V3)  at (1, 0, 2);   \coordinate (V4)  at (-1, 0, 2);
    \coordinate (V5)  at (0, 2, 1);   \coordinate (V6)  at (0, -2, 1);
    \coordinate (V7)  at (2, 0, 1);   \coordinate (V8)  at (-2, 0, 1);
    \coordinate (V9)  at (1, 2, 0);   \coordinate (V10) at (-1, 2, 0);
    \coordinate (V11) at (1, -2, 0);  \coordinate (V12) at (-1, -2, 0);
    \coordinate (V13) at (2, 1, 0);   \coordinate (V14) at (-2, 1, 0);
    \coordinate (V15) at (2, -1, 0);  \coordinate (V16) at (-2, -1, 0);
    \coordinate (V17) at (0, 2, -1);  \coordinate (V18) at (0, -2, -1);
    \coordinate (V19) at (2, 0, -1);  \coordinate (V20) at (-2, 0, -1);
    \coordinate (V21) at (0, 1, -2);  \coordinate (V22) at (0, -1, -2);
    \coordinate (V23) at (1, 0, -2);  \coordinate (V24) at (-1, 0, -2);

    \begin{scope}[dashed, very thick, fill opacity=0.3]
        \filldraw[fill=yellow!40] (V21)--(V23)--(V22)--(V24)--cycle; 
        \filldraw[fill=yellow!40] (V6)--(V11)--(V18)--(V12)--cycle;
        \filldraw[fill=yellow!40] (V8)--(V14)--(V20)--(V16)--cycle;
        \filldraw[fill=blue!30] (V22)--(V23)--(V19)--(V15)--(V11)--(V18)--cycle;
        \filldraw[fill=blue!30] (V22)--(V18)--(V12)--(V16)--(V20)--(V24)--cycle;
        \filldraw[fill=blue!30] (V21)--(V17)--(V10)--(V14)--(V20)--(V24)--cycle;
        \filldraw[fill=blue!30] (V2)--(V4)--(V8)--(V16)--(V12)--(V6)--cycle;
    \end{scope}

    \begin{scope}[very thick, fill opacity=0.7]
        \filldraw[fill=yellow!60] (V1)--(V3)--(V2)--(V4)--cycle;     
        \filldraw[fill=yellow!60] (V5)--(V9)--(V17)--(V10)--cycle;   
        \filldraw[fill=yellow!60] (V7)--(V13)--(V19)--(V15)--cycle;  
        \filldraw[fill=blue!50] (V1)--(V3)--(V7)--(V13)--(V9)--(V5)--cycle;
        \filldraw[fill=blue!50] (V1)--(V5)--(V10)--(V14)--(V8)--(V4)--cycle;
        \filldraw[fill=blue!50] (V2)--(V3)--(V7)--(V15)--(V11)--(V6)--cycle;
        \filldraw[fill=blue!50] (V21)--(V23)--(V19)--(V13)--(V9)--(V17)--cycle;
    \end{scope}

\end{tikzpicture}
    \end{subfigure}
    \caption{The truncated octahedron and one of its projections}
    \label{Fig: TruOct}
\end{figure}
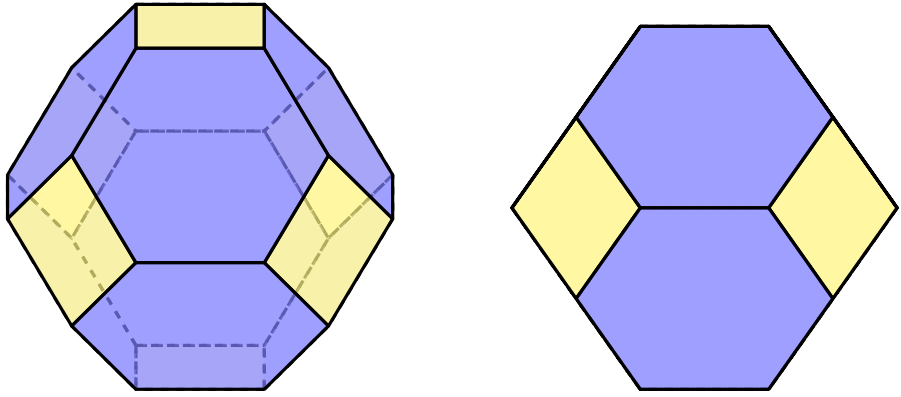

\begin{theorem}\label{The: Lower Bound}
    If $n= 16k^3-33k^2+24k-6$ for some $k\in \mathbb{N}$, then
    \[
    6n-6\sqrt[3]{2}n^\frac{2}{3}< 6n-6(8k^2-11k+4)\leq c_{A}(n)\leq c(n).
    \]
    
\end{theorem}

We end this section with a conjecture that states the lower bound in Theorem \ref{The: Lower Bound} is in fact an upper bound.

\begin{conjecture}
    For $n$ sufficiently large, if $n=16k^3-33k^2+24k-6$ for some $k\in \mathbb{N}$, then \[c(n)=6n-6(8k^2-11k+4).\]
    
\end{conjecture}

\section{Extremal contact graphs are minimally rigid}

In this section we prove Theorem \ref{The: Rigid}. 
It suffices to show that $c(n+1)\geq c(n)+3$ for all $n\geq 3$. 
To see why, suppose this holds. Then since $c(4)=6$ we immediately get $c(n)\geq 3n-6$ for all $n\geq 4$. Additionally, let $n\geq 3$ and consider an extremal packing of $n+1$ unit-diameter balls where one ball is in contact with at most two others.
Removing this ball creates a packing of $n$ unit-diameter balls with a contact number of at least $c(n+1)-2$, which shows $c(n+1)-2\leq c(n)$, a contradiction.

We prove $c(n+1)\geq c(n)+3$ for all $n\geq 3$ by induction on $n$. The base case holds as $c(4)=c(3)+3$. For the inductive step, suppose $c(n)\geq c(n-1) +3$ for some $n\geq 4$.
Consider an extremal packing of $n$ unit-diameter balls. Since $c(n)\geq c(n-1)+3$ each ball in the packing comes into contact with at least three others. 
Add to this packing an additional ball that is not in contact with any other ball.
Move this ball towards an arbitrary ball in the packing until it comes into contact with at least one other ball.
For each point $v\in \mathbb{R}^3$ let $B(v,r)$ denote the ball of radius $r$ centered at $v$, and let $B(v):=B(v,\frac{1}{2})$ denote the unit-diameter ball centered at $v$. 
Let $x$ be the center point of the ball we added to the packing, and let $y$ be the center point of the first ball it comes in contact with. 
If $B(x)$ is in contact with at least three balls we are done, so suppose $B(x)$ is in contact with at most two balls. Note that $x$ is at a distance at least one from every center point of a ball in our packing.

Without loss of generality we may assume $y$ is the origin. Then the center point of each ball that comes into contact with $B(y)$ is on $\mathbb{S}^2:=\bd{B(y,1)}$. 
Since we added a ball onto an extremal packing of $n$ balls, $B(y)$ was in contact with at least three other balls before $B(x)$ was added. 
This implies there are center points of balls in the packing, $z_1,z_2,\dots ,z_\ell$ where $\ell\geq 3$, which are distinct from $x$ and lie on $\mathbb{S}^2$.
For each $u\in \mathbb{S}^2$ and $r\in [0, 180]$, let $S(u,r^\circ)$ denote the spherical cap centered at $u$ with angular radius $r$.
Then for any $s\in \mathbb{S}^2$ and $i\in \{1,2,\dots, \ell\}$, $s$ is at a distance greater than one from $z_i$ if and only if $s\notin S(z_i,60^\circ)$, and $s$ is at a unit distance from $z_i$ if and only if $s\in \bd{S(z_i,60^\circ)}$.

Let $s_1,s_2,\dots, s_k$ where $k\geq0$ denote the center points of balls in the packing such that $1<\|s_i\|\leq 2$ and let $u_i=\frac{s_i}{\|s_i\|}$. For each $i$ let $S(u_i,r_i)$ be the spherical cap which is the intersection of $\mathbb{S}^2$ and $B(s_i,1)$, note that $r_i<60^\circ$. 
A point $s\in \mathbb{S}^2$ is a distance greater than one from $s_i$ if and only if $s\notin S(u_i,r_i)$, and $s$ is at a unit distance from $s_i$ if and only if $s\in \bd{S(u_i,r_i)}$. 

Given the spherical caps
\[S(z_1,60^\circ),S(z_2,60^\circ),\dots, S(z_\ell,60^\circ),S(u_1,r_1),S(u_2,r_2),\dots, S(u_k,r_k)\]
our goal is to reposition $x$ on $\mathbb{S}^2$ such that $x$ is on the boundary of at least two spherical caps and not in the interior of any spherical cap.
If we are able to reposition $x$ in such a way, by the previous two paragraphs, this would imply that $x$ is a unit distance from at least three center points of balls in the packing, and a distance greater than one from every other center point of a ball in the packing. 
The latter follows from $x$ not being in the interior of any of the spherical caps and the fact that since $x\in \mathbb{S}^2$, every point at a distance greater than $2$ from $y$ is also at a distance greater than $1$ from $x$.
This would imply the ball $B(x)$ is in contact with at least three others and has an interior disjoint from every other ball in the packing.
In other words, adding $B(x)$ to the packing creates a packing of $n+1$ unit-diameter balls that has a contact number at least $c(n)+3$ and hence $c(n+1)\geq c(n)+3$ completing the inductive step.

We claim that two of the spherical caps $S(z_1,60^\circ),S(z_2,60^\circ),S(z_3,60^\circ)$ must intersect.
This follows from the fact that for any three points on a sphere, some pair of these points must be at an angular distance at most $120^\circ$. 
This can be easily seen from the following. Let $z_1,z_2,z_3 \in \mathbb{S}^2$, then
\begin{align*}
0&\leq \|z_1+z_2+z_3\|^2=\langle z_1+z_2+z_3,z_1+z_2+z_3\rangle\\
&= 3+2\langle z_1,z_2\rangle +2\langle z_2,z_3\rangle+2\langle z_1,z_3\rangle.
\end{align*} 

\noindent This implies that $\max_{i\neq j} \langle z_i,z_j \rangle \geq -\frac{1}{2}$ and so the angle between the pair of points achieving this maximum is at most $120^\circ$.

Without loss of generality suppose $S(z_1,60^\circ)$ and $S(z_2,60^\circ)$ intersect. Since $x$ is on the boundary of at most one spherical cap and not in the interior of any spherical cap, it follows that the union, $\bigcup_{i=1}^\ell S(z_i,60^\circ) \cup \bigcup_{i=1}^k S(u_i,r_i)$, does not cover the sphere. This union will be a collection of connected components on $\mathbb{S}^2$, let $C$ be the connected component that contains $S(z_1,60^\circ)$ and $S(z_2,60^\circ)$. To simplify notation, we relabel the caps contained in $C$ so that $C=\cup_{i=1}^j S_i$ where $S_1$ and $S_2$ denote $S(z_1,60^\circ)$ and $S(z_2,60^\circ)$ respectively and $S_i$ for each $i\geq 3$ denote some other spherical cap in our collection. Since the union of all spherical caps does not cover $\mathbb{S}^2$, $C$ must have a boundary point $b$. The point $b$ must be on the boundary of at least one spherical cap and not in the interior of any cap. In the case $b$ is a boundary point of at least two spherical caps we are done by setting $x=b$. Otherwise $b$ is in the boundary of exactly one spherical cap, say $S_m$ for some $m=1,\dots, j$, note we may assume the angular radius of $S_m$ is positive as $b$ is in the boundary of exactly one spherical cap and $C$ is connected and not a singleton. We claim that the boundary of some cap in the list $S_1,S_2,\dots, S_{m-1},S_{m+1},\dots, S_j$ must intersect the boundary of $S_m$. To see this, suppose it is false, then for each cap $S_i$ where $i\neq m$ it must be the case that $\bd{S_m}$ is contained in the interior of $S_i$ or in the exterior of $S_i$. In the case $\bd{S_m}$ is contained in the interior of some $S_i$, this contradicts $b$ being a boundary point of $C$. In the case $\bd{S_m}$ is in the exterior of every $S_i$ then $C=S_m$ which contradicts that the angular radius of $S_m$ is at most $60^\circ$. This is due to $C$ containing $S_1$ and $S_2$ which are both spherical caps with angular radius $60^\circ$, so if a spherical cap contains both $S_1$ and $S_2$ it must have an angular radius greater than $60^\circ$. 

By the previous paragraph the set $\bigcup_{i\neq m} (\bd{S_m}\cap \bd{S_i})$ is non-empty. Since each cap has an angular radius at most $60^\circ$, $|\bd{S_m}\cap \bd{S_i}|\leq 2$ for each $i$. So $\bigcup_{i\neq m} (\bd{S_m}\cap \bd{S_i}) $ is a finite set of points on $\bd{S_m}$ not containing $b$. Choose the point $b'$ in this set that has smallest angular distance, measured from the center of $S_m$, from $b$. We claim that $b'$ is not in the interior of any spherical cap, if it were then since $b$ is in the exterior of every cap beside $S_m$, there must be an intersection of the boundary of this cap with both circular arcs connecting $b'$ and $b$ contradicting the choice of $b'$. Setting $x=b'$ completes the proof as $b'$ is on the boundary of at least two spherical caps and not in the interior of any.

\section{The contact number of small sphere packings}

In this section we prove Theorems \ref{The: Small Packings} and \ref{The: Enumeration}.
Let $G$ be the contact graph of a packing of $n$ unit-diameter balls in $\mathbb{R}^3$. 
We start by listing some known facts about the structure of $G$.

\begin{lemma}[Folklore and \cite{EmpAMB2011}]\label{Lem: Rules 1-4}
    $K_5$ and $K_{3,3}$ are forbidden subgraphs of $G$. Additionally, a pair of vertices has at most five common neighbors, and the induced subgraph on the common neighborhood of an adjacent pair of vertices is a subgraph of $P_5$.
\end{lemma}

That $K_5$ is a forbidden subgraph follows from the classical fact that an equidistant set of points in $\mathbb{R}^3$ has cardinality at most $4$. $K_{3,3}$ is a forbidden subgraph because the intersection of the boundaries of three unit-radius balls in $\mathbb{R}^3$ has at most two points. The last two parts of the lemma were shown in \cite{EmpAMB2011} as follows. The common neighborhood of a pair of vertices, $x$ and $y$, must lie on a circle which is the boundary intersection of unit-radius balls centered at $x$ and $y$. The radius of the circle is at its maximum when $x$ and $y$ are a unit distance apart. On this maximum radius circle only five vertices can be placed so that distance between any two is at least one. Additionally, the only time a pair of vertices on this circle is at a unit distance is when they are consecutive points on the circle. Lastly, they showed that when $x$ and $y$ are at a unit distance, if five points lie on the circle at a distance at least one from each other, then there must be a consecutive pair that is at a distance greater than one, or equivalently the circle's radius does not inscribe a unit-length regular pentagon.

The approach in \cite{EmpAMB2011} uses these rules and others to produce an enumeration of minimally rigid packings of unit-diameter balls that is likely to be complete.
The authors of \cite{EmpAMB2011} produced this by determining both the adjacency matrix and the distance matrix of $G$.
The distance matrix of $G$ has as its $(i,j)$ entry the Euclidean distance between the center points of the $i$th and $j$th ball.
Given a graph that one wishes to check is realizable or not as the contact graph of a packing of unit-diameter balls, determining the distance matrix (or at least part of it) can prove the graph is not realizable. 
For example, if $K_5$ minus an edge is a subgraph of $G$, then the distance between the pair of non-adjacent vertices is determined.
However, if some other rule determines this distance to be a different value, then the packing must not be realizable. 
Of course, if one uses floating-point arithmetic for these calculations, some packings may be ruled out erroneously. 

We utilize elements of the approach in \cite{EmpAMB2011} but do not calculate the distance matrix. Instead, we rely solely on the adjacency matrix of the contact graph.
Utilizing graph searches, for each $n=6,7,8,$ and $9$, we will create a list of $n$-vertex graphs that will include all $n$-vertex contact graphs with at least $3n-6$ edges. 
This list may include some graphs that are not realizable as the contact graph of a packing of unit-diameter balls.
However, since we only need an upper bound on the number of edges in such graphs this will not be a problem as long as all graphs we generate do not have more than $3n-6$ edges.
It turns out that the lists we generate match the potential enumerations found in \cite{EmpAMB2011} for $n=6,7,$ and $8$. 
This implies the list of contact graphs we generate enumerates all extremal contact graphs for $n=6,7,$ and $8$. 
In order to accomplish this we need to know more about the combinatorial structure of contact graphs so that we can rule out the graphs with more than $3n-6$ edges.
We do this by creating a list of forbidden subgraphs of $G$.

\begin{lemma}
    Each graph in Figure \ref{fig:forbidden} is a forbidden subgraph of a contact graph of unit-diameter balls.
\end{lemma}

\begin{figure}[htbp]
  \centering
  \begin{subfigure}[t]{0.31\textwidth}
  \centering
  \begin{tikzpicture}[line cap=round, line join=round]
    \coordinate (a) at (90.000:1.300);
    \coordinate (b) at (30.000:1.300);
    \coordinate (c) at (-30.000:1.300);
    \coordinate (d) at (-90.000:1.300);
    \coordinate (e) at (-150.000:1.300);
    \coordinate (f) at (-210.000:1.300);
    \draw (a)--(b) (a)--(c) (a)--(d) (b)--(c) (b)--(d) (c)--(d) 
          (e)--(a) (e)--(b) (f)--(c) (f)--(d) (e)--(f);
    \fill (a) circle (1.5pt);
    \fill (b) circle (1.5pt);
    \fill (c) circle (1.5pt);
    \fill (d) circle (1.5pt);
    \fill (e) circle (1.5pt);
    \fill (f) circle (1.5pt);
    \node at (90.000:1.660) {$a$};
    \node at (30.000:1.660) {$b$};
    \node at (-30.000:1.660) {$c$};
    \node at (-90.000:1.660) {$d$};
    \node at (-150.000:1.660) {$e$};
    \node at (-210.000:1.660) {$f$};
  \end{tikzpicture}
  \caption{$H_1$}
  \end{subfigure}\hfill
  \begin{subfigure}[t]{0.31\textwidth}
  \centering
  \begin{tikzpicture}[line cap=round, line join=round]
    \coordinate (a) at (90.000:1.300);
    \coordinate (b) at (38.571:1.300);
    \coordinate (c) at (-12.857:1.300);
    \coordinate (d) at (-64.286:1.300);
    \coordinate (e) at (-115.714:1.300);
    \coordinate (f) at (-167.143:1.300);
    \coordinate (g) at (-218.571:1.300);
    \draw (a)--(b) (a)--(c) (a)--(d) (b)--(c) (b)--(d) (c)--(d) 
          (e)--(a) (e)--(b) (f)--(c) (f)--(d) (g)--(e) (g)--(f);
    \fill (a) circle (1.5pt);
    \fill (b) circle (1.5pt);
    \fill (c) circle (1.5pt);
    \fill (d) circle (1.5pt);
    \fill (e) circle (1.5pt);
    \fill (f) circle (1.5pt);
    \fill (g) circle (1.5pt);
    \node at (90.000:1.660) {$a$};
    \node at (38.571:1.660) {$b$};
    \node at (-12.857:1.660) {$c$};
    \node at (-64.286:1.660) {$d$};
    \node at (-115.714:1.660) {$e$};
    \node at (-167.143:1.660) {$f$};
    \node at (-218.571:1.660) {$g$};
  \end{tikzpicture}
  \caption{$H_2$}
  \end{subfigure}\hfill
  \begin{subfigure}[t]{0.31\textwidth}
  \centering
  \begin{tikzpicture}[line cap=round, line join=round]
    \coordinate (a) at (90.000:1.300);
    \coordinate (b) at (38.571:1.300);
    \coordinate (c) at (-12.857:1.300);
    \coordinate (d) at (-64.286:1.300);
    \coordinate (e) at (-115.714:1.300);
    \coordinate (f) at (-167.143:1.300);
    \coordinate (g) at (-218.571:1.300);
    \draw (a)--(b) (a)--(c) (a)--(d) (b)--(c) (b)--(d) (c)--(d) 
          (e)--(b) (e)--(c) (f)--(c) (f)--(d) (g)--(a) (g)--(e) 
          (g)--(f);
    \fill (a) circle (1.5pt);
    \fill (b) circle (1.5pt);
    \fill (c) circle (1.5pt);
    \fill (d) circle (1.5pt);
    \fill (e) circle (1.5pt);
    \fill (f) circle (1.5pt);
    \fill (g) circle (1.5pt);
    \node at (90.000:1.660) {$a$};
    \node at (38.571:1.660) {$b$};
    \node at (-12.857:1.660) {$c$};
    \node at (-64.286:1.660) {$d$};
    \node at (-115.714:1.660) {$e$};
    \node at (-167.143:1.660) {$f$};
    \node at (-218.571:1.660) {$g$};
  \end{tikzpicture}
  \caption{$H_3$}
  \end{subfigure}
  \\[2ex]
  \begin{subfigure}[t]{0.31\textwidth}
  \centering
  \begin{tikzpicture}[line cap=round, line join=round]
    \coordinate (a) at (90.000:1.300);
    \coordinate (b) at (45.000:1.300);
    \coordinate (c) at (0.000:1.300);
    \coordinate (d) at (-45.000:1.300);
    \coordinate (e) at (-90.000:1.300);
    \coordinate (f) at (-135.000:1.300);
    \coordinate (g) at (-180.000:1.300);
    \coordinate (h) at (-225.000:1.300);
    \draw (a)--(b) (a)--(c) (a)--(d) (a)--(e) (g)--(b) (g)--(c) 
          (g)--(d) (g)--(e) (b)--(c) (c)--(d) (d)--(e) (f)--(a) 
          (f)--(b) (f)--(e) (h)--(c) (h)--(d) (h)--(f);
    \fill (a) circle (1.5pt);
    \fill (b) circle (1.5pt);
    \fill (c) circle (1.5pt);
    \fill (d) circle (1.5pt);
    \fill (e) circle (1.5pt);
    \fill (f) circle (1.5pt);
    \fill (g) circle (1.5pt);
    \fill (h) circle (1.5pt);
    \node at (90.000:1.660) {$a$};
    \node at (45.000:1.660) {$b$};
    \node at (0.000:1.660) {$c$};
    \node at (-45.000:1.660) {$d$};
    \node at (-90.000:1.660) {$e$};
    \node at (-135.000:1.660) {$f$};
    \node at (-180.000:1.660) {$g$};
    \node at (-225.000:1.660) {$h$};
  \end{tikzpicture}
  \caption{$H_4$}
  \end{subfigure}\hfill
  \begin{subfigure}[t]{0.31\textwidth}
  \centering
  \begin{tikzpicture}[line cap=round, line join=round]
    \coordinate (a) at (90.000:1.300);
    \coordinate (b) at (45.000:1.300);
    \coordinate (c) at (0.000:1.300);
    \coordinate (d) at (-45.000:1.300);
    \coordinate (e) at (-90.000:1.300);
    \coordinate (f) at (-135.000:1.300);
    \coordinate (g) at (-180.000:1.300);
    \coordinate (h) at (-225.000:1.300);
    \draw (a)--(b) (a)--(c) (a)--(d) (a)--(e) (g)--(b) (g)--(c) 
          (g)--(d) (g)--(e) (b)--(c) (c)--(d) (d)--(e) (f)--(a) 
          (f)--(b) (f)--(e) (h)--(b) (h)--(d) (h)--(f);
    \fill (a) circle (1.5pt);
    \fill (b) circle (1.5pt);
    \fill (c) circle (1.5pt);
    \fill (d) circle (1.5pt);
    \fill (e) circle (1.5pt);
    \fill (f) circle (1.5pt);
    \fill (g) circle (1.5pt);
    \fill (h) circle (1.5pt);
    \node at (90.000:1.660) {$a$};
    \node at (45.000:1.660) {$b$};
    \node at (0.000:1.660) {$c$};
    \node at (-45.000:1.660) {$d$};
    \node at (-90.000:1.660) {$e$};
    \node at (-135.000:1.660) {$f$};
    \node at (-180.000:1.660) {$g$};
    \node at (-225.000:1.660) {$h$};
  \end{tikzpicture}
  \caption{$H_5$}
  \end{subfigure}\hfill
  \begin{subfigure}[t]{0.31\textwidth}
  \centering
  \begin{tikzpicture}[line cap=round, line join=round]
    \coordinate (a) at (90.000:1.300);
    \coordinate (b) at (38.571:1.300);
    \coordinate (c) at (-12.857:1.300);
    \coordinate (d) at (-64.286:1.300);
    \coordinate (e) at (-115.714:1.300);
    \coordinate (f) at (-167.143:1.300);
    \coordinate (g) at (-218.571:1.300);
    \draw (a)--(b) (a)--(c) (a)--(d) (b)--(c) (b)--(d) (c)--(d) 
          (e)--(b) (e)--(c) (e)--(d) (f)--(a) (f)--(b) (f)--(c) 
          (g)--(c) (g)--(e) (g)--(f);
    \fill (a) circle (1.5pt);
    \fill (b) circle (1.5pt);
    \fill (c) circle (1.5pt);
    \fill (d) circle (1.5pt);
    \fill (e) circle (1.5pt);
    \fill (f) circle (1.5pt);
    \fill (g) circle (1.5pt);
    \node at (90.000:1.660) {$a$};
    \node at (38.571:1.660) {$b$};
    \node at (-12.857:1.660) {$c$};
    \node at (-64.286:1.660) {$d$};
    \node at (-115.714:1.660) {$e$};
    \node at (-167.143:1.660) {$f$};
    \node at (-218.571:1.660) {$g$};
  \end{tikzpicture}
  \caption{$H_6$}
  \end{subfigure}
  \\[2ex]
  \begin{subfigure}[t]{0.31\textwidth}
  \centering
  \begin{tikzpicture}[line cap=round, line join=round]
    \coordinate (a) at (90.000:1.300);
    \coordinate (b) at (45.000:1.300);
    \coordinate (c) at (0.000:1.300);
    \coordinate (d) at (-45.000:1.300);
    \coordinate (e) at (-90.000:1.300);
    \coordinate (f) at (-135.000:1.300);
    \coordinate (g) at (-180.000:1.300);
    \coordinate (h) at (-225.000:1.300);
    \draw (a)--(b) (b)--(c) (c)--(d) (d)--(e) (e)--(a) (f)--(a) 
          (f)--(b) (f)--(c) (f)--(d) (f)--(e) (g)--(a) (g)--(b) 
          (g)--(c) (g)--(d) (g)--(e) (h)--(a) (h)--(c) (h)--(f);
    \fill (a) circle (1.5pt);
    \fill (b) circle (1.5pt);
    \fill (c) circle (1.5pt);
    \fill (d) circle (1.5pt);
    \fill (e) circle (1.5pt);
    \fill (f) circle (1.5pt);
    \fill (g) circle (1.5pt);
    \fill (h) circle (1.5pt);
    \node at (90.000:1.660) {$a$};
    \node at (45.000:1.660) {$b$};
    \node at (0.000:1.660) {$c$};
    \node at (-45.000:1.660) {$d$};
    \node at (-90.000:1.660) {$e$};
    \node at (-135.000:1.660) {$f$};
    \node at (-180.000:1.660) {$g$};
    \node at (-225.000:1.660) {$h$};
  \end{tikzpicture}
  \caption{$H_7$}
  \end{subfigure}\hfill
  \begin{subfigure}[t]{0.31\textwidth}
  \centering
  \begin{tikzpicture}[line cap=round, line join=round]
    \coordinate (a) at (90.000:1.300);
    \coordinate (b) at (45.000:1.300);
    \coordinate (c) at (0.000:1.300);
    \coordinate (d) at (-45.000:1.300);
    \coordinate (e) at (-90.000:1.300);
    \coordinate (f) at (-135.000:1.300);
    \coordinate (g) at (-180.000:1.300);
    \coordinate (h) at (-225.000:1.300);
    \draw (a)--(b) (b)--(c) (c)--(d) (d)--(a) (e)--(a) (e)--(b) 
          (e)--(c) (e)--(d) (f)--(a) (f)--(b) (f)--(c) (f)--(d) 
          (g)--(a) (g)--(b) (g)--(f) (h)--(a) (h)--(b) (h)--(f);
    \fill (a) circle (1.5pt);
    \fill (b) circle (1.5pt);
    \fill (c) circle (1.5pt);
    \fill (d) circle (1.5pt);
    \fill (e) circle (1.5pt);
    \fill (f) circle (1.5pt);
    \fill (g) circle (1.5pt);
    \fill (h) circle (1.5pt);
    \node at (90.000:1.660) {$a$};
    \node at (45.000:1.660) {$b$};
    \node at (0.000:1.660) {$c$};
    \node at (-45.000:1.660) {$d$};
    \node at (-90.000:1.660) {$e$};
    \node at (-135.000:1.660) {$f$};
    \node at (-180.000:1.660) {$g$};
    \node at (-225.000:1.660) {$h$};
  \end{tikzpicture}
  \caption{$H_8$}
  \end{subfigure}\hfill
  \begin{subfigure}[t]{0.31\textwidth}
  \centering
  \begin{tikzpicture}[line cap=round, line join=round]
    \coordinate (a) at (90.000:1.300);
    \coordinate (b) at (50.000:1.300);
    \coordinate (c) at (10.000:1.300);
    \coordinate (d) at (-30.000:1.300);
    \coordinate (e) at (-70.000:1.300);
    \coordinate (f) at (-110.000:1.300);
    \coordinate (g) at (-150.000:1.300);
    \coordinate (h) at (-190.000:1.300);
    \coordinate (i) at (-230.000:1.300);
    \draw (a)--(b) (b)--(c) (c)--(d) (d)--(e) (f)--(a) (f)--(b) 
          (f)--(c) (f)--(d) (f)--(e) (g)--(a) (g)--(b) (g)--(c) 
          (g)--(d) (g)--(e) (f)--(g) (h)--(a) (h)--(e) (h)--(f) 
          (i)--(a) (i)--(e) (i)--(g) (i)--(h);
    \fill (a) circle (1.5pt);
    \fill (b) circle (1.5pt);
    \fill (c) circle (1.5pt);
    \fill (d) circle (1.5pt);
    \fill (e) circle (1.5pt);
    \fill (f) circle (1.5pt);
    \fill (g) circle (1.5pt);
    \fill (h) circle (1.5pt);
    \fill (i) circle (1.5pt);
    \node at (90.000:1.660) {$a$};
    \node at (50.000:1.660) {$b$};
    \node at (10.000:1.660) {$c$};
    \node at (-30.000:1.660) {$d$};
    \node at (-70.000:1.660) {$e$};
    \node at (-110.000:1.660) {$f$};
    \node at (-150.000:1.660) {$g$};
    \node at (-190.000:1.660) {$h$};
    \node at (-230.000:1.660) {$i$};
  \end{tikzpicture}
  \caption{$H_9$}
  \end{subfigure}
  \caption{Forbidden subgraphs of a contact graph of unit-diameter balls, each graph is depicted with the same vertex labeling used in Lemmas \ref{Lem: Forbidding 1 and 2} through \ref{Lem: Forbidding 9}}
  \label{fig:forbidden}
\end{figure}
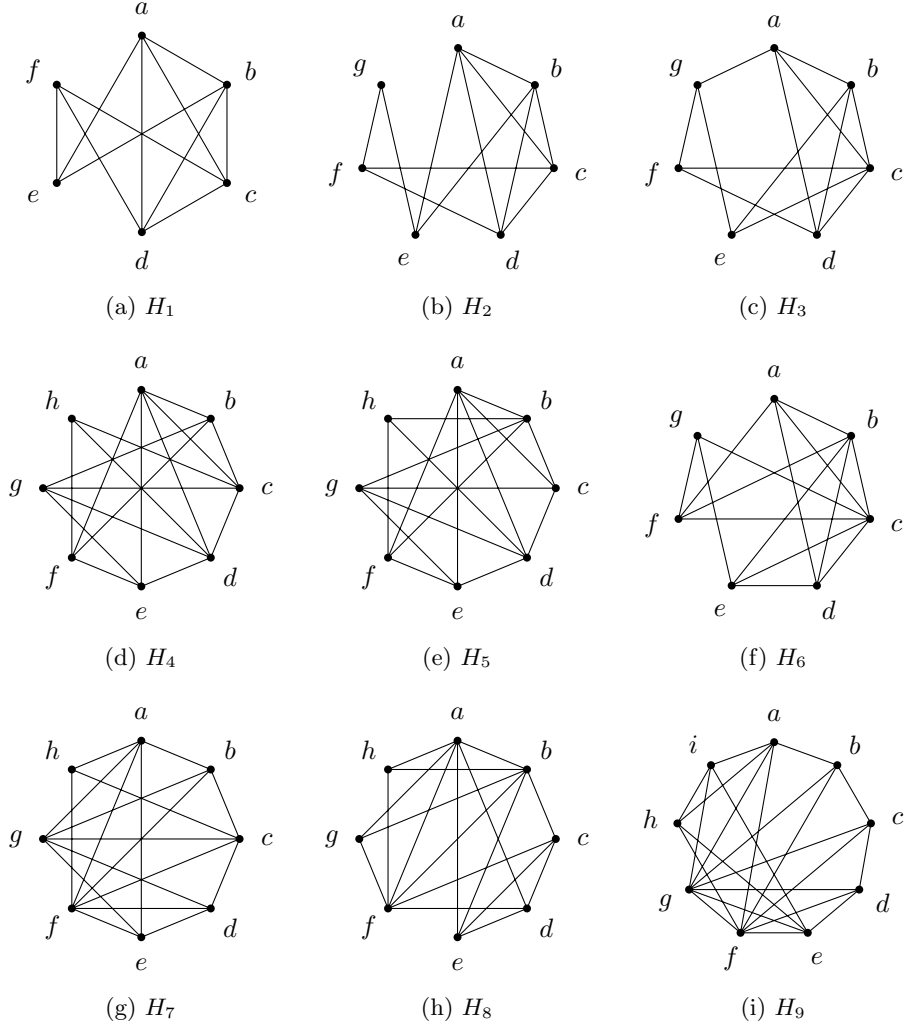

The proof of this lemma is given in Lemmas \ref{Lem: Forbidding 1 and 2}–\ref{Lem: Forbidding 9}. These forbidden subgraphs will suffice for our graph searches to successfully enumerate the extremal contact graphs for $n=6,7,$ and $8$ and bound the contact number for $n=6,7,8,$ and $9$. 
In the following series of lemmas, we show each $n$-vertex graph $H$ is a forbidden subgraph by showing the non-existence of a set of $n$ points $v_1,v_2,\dots, v_n \in \mathbb{R}^3$ where $v_iv_j\in E(H)\implies \|v_i-v_j\|=1$ and $v_iv_j\notin E(H)\implies \|v_i-v_j\|\geq 1$.
We show non-existence of such a set in one of two ways.
The first approach is to use a subset of the distance constraints to parameterize the coordinates of all the vertices of the embedded contact subgraph $H$ (up to rotation and translation).
With this parameterization, we then show that over the parameterized domain a pair of vertices that has an edge in $H$ is always at a distance greater than one from each other, hence no set of points can satisfy all distance constraints simultaneously.
The second approach is to use the distance constraints to parameterize the coordinates of all but one of the vertices of $H$. This last vertex will need to be at a unit distance from some of the vertices and a distance at least one from the others. We show over the parameterized domain that every point in $\mathbb{R}^3$ that is a unit distance from the required vertices, must be at a distance less than one from some other vertex.

To accomplish this last step in each approach we use interval arithmetic, as implemented by Arb via python-flint \cites{Flint,Arb}. In interval arithmetic each real quantity is represented not by an approximate floating-point number, but by a closed interval with floating-point endpoints that is guaranteed to contain the quantity. Every arithmetic operation and elementary function is implemented so that the resulting output interval contains the result of the operation applied to any value in the input interval. This is achieved by rounding each operation outward, the lower endpoint down and the upper endpoint up, so that the interval can only widen. In our application we evaluate, over a domain of parameters, a distance between two points and compare it with 1. To certify such a strict inequality it suffices to compute an interval for the distance that lies entirely on one side of 1, its lower endpoint above 1 when the distance must exceed 1, or its upper endpoint below 1 when the distance must be less than 1. A single evaluation over the whole parameter domain typically yields an interval too wide to obtain the strict inequality, so we instead partition the domain into finitely many boxes, which we call cells, and obtain the strict inequality over each cell. The strict inequality is then verified when it holds on every cell. 

One restriction on interval arithmetic is that the quantity desired to be bounded needs to be made up by a composition of not too complicated functions. This is to ensure that the implementation can produce, from an input interval, an output interval containing the image of the input under the function. By repeated application for each composing function one is able to bound the desired quantity. A full list of the functions that Arb via python-flint can use in interval arithmetic is found in the documentation \cite{Flint}. For our applications we only need the very basic functions $+,-,\times, \div, \sqrt{\phantom{x}},\cos(\phantom{x})\sin(\phantom{x}), \max(\phantom{x})$ and $\min(\phantom{x})$. It is important to note that due to this outward rounding process, interval arithmetic does not, in general, produce a tight bound for a desired quantity. However, for our purposes, the distances we want to bound away from $1$ do not come close to $1$ over the parameterized domain, which enables a coarse approach like interval arithmetic to succeed. One could, in theory, prove each of these lemmas by hand, and some of them are simple enough to do so. However, since interval arithmetic is required to greatly simplify the justification of some of them, we elect to use interval arithmetic in all cases as, with the machinery already in place, this simplifies the analysis. The interval arithmetic for each lemma can be found in the file named \texttt{Small\_sphere\_packings.py} and is attached to the arXiv version of this paper.

\begin{figure}
    \centering
    \begin{subfigure}[t]{0.55\textwidth}
\tdplotsetmaincoords{77}{100}

\begin{tikzpicture}[tdplot_main_coords, scale=2.2,
                    line cap=round, line join=round]

  \coordinate (a) at (-0.5, 0, 0);
  \coordinate (b) at  (0.5, 0, 0);
  \coordinate (c) at (0,  0.7071,  0.5);
  \coordinate (d) at (0,  0.7071, -0.5);

  \draw[thick] (a) -- (b);
  \draw[thick] (a) -- (c);
  \draw[thick] (a) -- (d);
  \draw[thick] (b) -- (c);
  \draw[thick] (b) -- (d);
  \draw[thick] (c) -- (d);

  \draw[dashed, thick]
    plot[domain=105.79:254.21, samples=80, variable=\t]
         ({0}, {0.8660*cos(\t)}, {0.8660*sin(\t)});

  \draw[dashed, thick]
    plot[domain=15.79:164.21, samples=80, variable=\s]
         ({0.8660*cos(\s)}, {0.7071 + 0.8660*sin(\s)}, {0});

  \foreach \v in {a,b,c,d} \filldraw (\v) circle (0.6pt);
  \node[anchor=south east] at (a) {$a$};
  \node[anchor=south] at (b) {$b$};
  \node[anchor=south]      at (c) {$c$};
  \node[anchor=north]      at (d) {$d$};

  \coordinate (e) at (0, -0.8660, 0);
  \filldraw (e) circle (0.6pt);
  \node[anchor=east] at (e) {$e$};
  \draw[thick] (e) -- (a);
  \draw[thick] (e) -- (b);

  \coordinate (f) at (0, 1.5731, 0);
  \filldraw (f) circle (0.6pt);
  \node[anchor=west] at (f) {$f$};
  \draw[thick] (f) -- (c);
  \draw[thick] (f) -- (d);

\end{tikzpicture}
\subcaption{Subgraph of $H_1$ and $H_2$}
\label{Fig: Forbidden 1-3a}
    \end{subfigure}
    \hfill
    \begin{subfigure}[t]{0.4\textwidth}

\tdplotsetmaincoords{72}{220}
\begin{tikzpicture}[tdplot_main_coords, scale=2.2,
                    line cap=round, line join=round]

  \coordinate (a) at (0,        0,        0.8165);
  \coordinate (b) at (0.5774,   0,        0);
  \coordinate (c) at (-0.2887,  0.5,      0);
  \coordinate (d) at (-0.2887, -0.5,      0);

  \draw[thick] (a) -- (b);
  \draw[thick] (a) -- (c);
  \draw[thick] (a) -- (d);
  \draw[thick] (b) -- (c);
  \draw[thick] (b) -- (d);
  \draw[thick] (c) -- (d);

  \draw[dashed, thick]
    plot[domain=250.53:398.94, samples=80, variable=\t]
         ({0.1443 + 0.8660*cos(\t)*0.5},
          {0.25   + 0.8660*cos(\t)*0.8660},
          {        0.8660*sin(\t)});

  \coordinate (e) at (0.3943, 0.6830, -0.7071);
  \filldraw (e) circle (0.6pt);
  \node[anchor=north east] at (e) {$e$};
  \draw[thick] (e) -- (b);
  \draw[thick] (e) -- (c);

  \draw[dashed, thick]
    plot[domain=141.05:289.47, samples=80, variable=\s]
         ({-0.2887 + 0.8660*cos(\s)},
          {                            0},
          {          0.8660*sin(\s)});

  \coordinate (f) at (-0.9958, 0, -0.5000);
  \filldraw (f) circle (0.6pt);
  \node[anchor=north west] at (f) {$f$};
  \draw[thick] (f) -- (c);
  \draw[thick] (f) -- (d);

  \foreach \v in {a,b,c,d} \filldraw (\v) circle (0.6pt);
  \node[anchor=south]      at (a) {$a$};
  \node[anchor=east]       at (b) {$b$};
  \node[anchor=south west] at (c) {$c$};
  \node[anchor=south]      at (d) {$d$};

\end{tikzpicture}
    \subcaption{Subgraph of $H_3$}
    \label{Fig: Forbidden 1-3b}
    \end{subfigure}
    \caption{Depiction of the embedded contact subgraphs of $H_1,H_2,$ and $H_3$, the dashed arcs represent the ranges $e$ and $f$ can be positioned in}
    \label{Fig: Forbidden 1-3}
\end{figure}
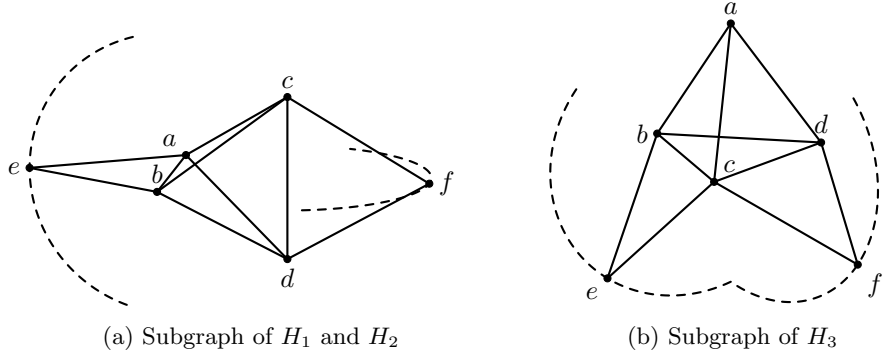

\begin{lemma}[Forbidding $H_1$ and $H_2$]\label{Lem: Forbidding 1 and 2}
    Let $a=(0,0,\frac{1}{2})$, $b=(0,0,-\frac{1}{2})$, $c=(\frac{1}{\sqrt2},\frac{1}{2},0)$ and $d=(\frac{1}{\sqrt2}, -\frac{1}{2},0)$ be the vertices of the regular unit-length tetrahedron. Suppose a point $e$ is at a unit distance from $a$ and $b$. Suppose a point $f$ is at a unit distance from $c$ and $d$. Suppose all unspecified distances between points are at least one (see Figure \ref{Fig: Forbidden 1-3a}). Then $e$ and $f$ are at a distance greater than one, and any point, $g$, at a unit distance from both $e$ and $f$ is at a distance less than $1$ from $a,b,c,$ or $d$.
\end{lemma}

\begin{proof}
    The point $e$ is parameterized by $e(s)=(\frac{\sqrt3}{2}\cos(s), \frac{\sqrt3}{2}\sin(s),0)$ for $s\in [0,2\uppi]$.
    Similarly $f$ is parameterized by $f(t)=(\frac{1}{\sqrt{2}}+\frac{\sqrt3}{2}\cos(t),0, \frac{\sqrt3}{2}\sin(t))$ for $t\in [0,2\uppi]$. To prove the first part of the claim we show that for any $s$ and $t$ such that $e(s)$ and $f(t)$ are at least a unit distance away from $a,b,c,$ and $d$, the distance between $e(s)$ and $f(t)$ is greater than $1$. 
    To show this we use interval arithmetic and separate our domain $[0,2\uppi]\times [0,2\uppi]$ into $240^2$ cells. 
    In each cell we check if at least one point, $(s,t)$, in the cell may have the property that both $e(s)$ and $f(t)$ are a distance at least one from $a,b,c,$ and $d$, in which case we compute a lower bound on the minimum distance between $e(s)$ and $f(t)$ over the cell and confirm it is greater than $1$. 

    Denote by $m_{x,y}$ the midpoint of the line segment connecting $x$ and $y$.
    For the second part of the claim, consider a point $g$ at a unit distance from both $e(s)$ and $f(t)$ for some $s,t\in [0,2\uppi]$.
    Such a point lies on a circle centered at $m_{e,f}$ with a radius of $\rho=\sqrt{1-\|e(s)-f(t)\|^2/4}$.
    Using the parallelogram law we obtain, for $x,y\in \{a,b,c,d\}$ that $\min\{\|g-x\|^2,\|g-y\|^2\}\leq \|g-m_{x,y}\|^2 +\frac{1}{4}$.
    With this and the triangle inequality we have
    \begin{align*}
    \min_{z\in \{a,b,c,d\}} \|g-z\|^2
    &\leq \min_{x\in \{a,b\}, y\in \{c,d\}} \|g-m_{x,y}\|^2+\frac{1}{4}
    \\&\leq \min_{x\in \{a,b\}, y\in \{c,d\}} (\|m_{e,f}-m_{x,y}\|+\rho)^2+\frac{1}{4}.
    \end{align*}
    Again we split up the domain of our parameters $s$ and $t$, $[0,2\uppi]\times [0,2\uppi]$, into $240^2$ cells and use interval arithmetic to verify that whenever $e(s)$ and $f(t)$ are a distance at least one form $a,b,c,$ and $d$ then $\min_{x\in \{a,b\}, y\in \{c,d\}} \|m_{e,f}-m_{x,y}\|+\rho <\sqrt{\frac{3}{4}}$.
    This completes the proof.
\end{proof}

\begin{lemma}[Forbidding $H_3$]\label{Lem: Forbidding 3}
    Suppose $a=(0,0,0)$, $b=({\frac{1}{\sqrt3}},0,\sqrt{\frac{2}{3}})$, $c=(-\frac{1}{2\sqrt3},\frac{1}{2},\sqrt{\frac{2}{3}})$ and $d=(-\frac{1}{2\sqrt3},-\frac{1}{2},\sqrt{\frac{2}{3}})$ be the vertices of the regular unit-length tetrahedron. Suppose a point $e$ is at a unit distance from $b$ and $c$. Suppose a point $f$ is at a unit distance from $c$ and $d$. 
    Suppose all unspecified distances between points are at least one (see Figure \ref{Fig: Forbidden 1-3b}).
    Then any point $g$ at a unit distance from $a,e,$ and $f$, is at a distance less than one from $b,c,$ or $d$.
\end{lemma}

\begin{proof}
    Suppose $a$ through $g$ are as stated in the Lemma. Then $g$ is on the unit sphere and the position of $g$ determines both $e$ and $f$. The reason is that since $e$ is at a unit distance from $b,c,$ and $g$ it can be one of two points, since $a$ is one of them, the point $e$ will be the reflection of $a$ about the plane going through $b,c,$ and $g$. Similarly, $f$ is the reflection of $a$ about the plane going through $c,d,$ and $g$. 
    Let $n_1=(c-b)\times (g-b)$ then $\|e-a\|\geq 1 \iff |c\cdot \frac{n_1}{\|n_1\|}|\geq \frac{1}{2}\iff (c\cdot n_1)^2 \geq \frac{\|n_1\|^2}{4}$. Similarly, if $n_2=(c-d)\times (g-d)$ then $\|f-a\|\geq 1 \iff (c\cdot n_2)^2 \geq \frac{\|n_2\|^2}{4}$.
    Additionally we have 
    \[
    \|e-d\|^2=\|(2\frac{n_1}{\|n_1\|}\cdot c)\frac{n_1}{\|n_1\|}-d\|^2=4(\frac{n_1}{\|n_1\|}\cdot c)^2-4(\frac{n_1}{\|n_1\|}\cdot c)(\frac{n_1}{\|n_1\|}\cdot d)+1.
    \]
    So the condition $\|e-d\| \geq 1$ is equivalent to 
    \[
    4(\frac{n_1}{\|n_1\|}\cdot c)^2-4(\frac{n_1}{\|n_1\|}\cdot c)(\frac{n_1}{\|n_1\|}\cdot d)\geq 0\iff 
    (\frac{n_1}{\|n_1\|}\cdot c)(\frac{n_1}{\|n_1\|}\cdot c-\frac{n_1}{\|n_1\|}\cdot d)\geq 0\]
    \[\iff (n_1\cdot c)(n_1\cdot c-n_1\cdot d)\geq 0.
    \]
    Similarly the condition $\|f-b\|\geq 1$ is equivalent to $(n_2\cdot c)(n_2\cdot c-n_2\cdot b)\geq 0$.
    The condition that $g$ is at least a distance of $1$ away from $b,c,$ and $d$ is equivalent to $\max_{x\in \{b,c,d\}} g\cdot x\leq \frac{1}{2}$.

    We parameterize $g$ using spherical polar coordinates for $(\theta, \phi)\in[0,\uppi]\times [0,2\uppi]$ by $g=(\sin(\theta) \cos(\phi), \sin(\theta) \sin(\phi), \cos(\theta))$.
    We split the unit sphere into $240^2$ cells and use interval arithmetic to verify on each cell if a point in the cell, $(\theta, \phi)$, which parametrizes $g$ could satisfy $(c\cdot n_1)^2 \geq \frac{\|n_1\|^2}{4}$, $(c\cdot n_2)^2 \geq \frac{\|n_2\|^2}{4}$, $(n_1\cdot c)(n_1\cdot c-n_1\cdot d)\geq 0$ and $(n_2\cdot c)(n_2\cdot c-n_2\cdot b)\geq 0$. 
    If such a point in the cell satisfying these inequalities could exist, then we compute $\max_{x\in \{b,c,d\}} g\cdot x$ over the cell and verify it is greater than $\frac{1}{2}$.
    This completes the proof.
\end{proof}

\begin{figure}
    \centering
    \begin{subfigure}[t]{0.5\textwidth}
\tdplotsetmaincoords{80}{85}
\begin{tikzpicture}[tdplot_main_coords, scale=3,
                    line cap=round, line join=round]

  \coordinate (a) at (0, 0, 0);
  \coordinate (b) at (0.83516465442, 0, 0.55);
  \coordinate (c) at (0.23648031434, 0.80098505662, 0.55);
  \coordinate (d) at (-0.70124390294, 0.45360444067, 0.55);
  \coordinate (e) at (-0.63360051743, -0.54410512248, 0.55);
  \coordinate (f) at (0.31298937593, -0.84488801066, 0.43382242912);
  \coordinate (g) at (0, 0, 1.1);

  \draw[dashed,thick] (a) -- (g);

  \draw[thick] (a) -- (b);
  \draw[thick] (a) -- (c);
  \draw[thick] (a) -- (d);
  \draw[thick] (a) -- (e);

  \draw[thick] (g) -- (b);
  \draw[thick] (g) -- (c);
  \draw[thick] (g) -- (d);
  \draw[thick] (g) -- (e);

  \draw[thick] (b) -- (c);
  \draw[thick] (c) -- (d);
  \draw[thick] (d) -- (e);

  \draw[thick] (f) -- (a);
  \draw[thick] (f) -- (b);
  \draw[thick] (f) -- (e);

  \foreach \v in {a,b,c,d,e,f,g} \filldraw (\v) circle (0.6pt);

  \node[anchor=north]      at (a) {$a$};
  \node[anchor=south west] at (b) {$b$};
  \node[anchor=west] at (c) {$c$};
  \node[anchor=north west] at (d) {$d$};
  \node[anchor=south east] at (e) {$e$};
  \node[anchor=east]       at (f) {$f$};
  \node[anchor=south]      at (g) {$g$};
\end{tikzpicture}
\subcaption{Subgraph of $H_4$ and $H_5$}
\label{Fig: Forbidden 4-6a }
    \end{subfigure}
    \hfill
    \begin{subfigure}[t]{0.4\textwidth}

\tdplotsetmaincoords{72}{220}
\begin{tikzpicture}[tdplot_main_coords, scale=3,
                    line cap=round, line join=round]

  \coordinate (a) at (0,        0,        0.8165);
  \coordinate (b) at (0.5774,   0,        0);
  \coordinate (c) at (-0.2887,  0.5,      0);
  \coordinate (d) at (-0.2887, -0.5,      0);

  \coordinate (e) at (0,        0,       -0.8165);
  \coordinate (f) at (0.480,    0.833,    0.544);

  \draw[thick] (a) -- (b);
  \draw[thick] (a) -- (c);
  \draw[thick] (a) -- (d);
  \draw[thick] (b) -- (c);
  \draw[thick] (b) -- (d);
  \draw[thick] (c) -- (d);

  \draw[thick] (e) -- (b);
  \draw[thick] (e) -- (c);
  \draw[thick] (e) -- (d);

  \draw[thick] (f) -- (a);
  \draw[thick] (f) -- (b);
  \draw[thick] (f) -- (c);

  \foreach \v in {a,b,c,d,e,f} \filldraw (\v) circle (0.6pt);

  \node[anchor=south]      at (a) {$a$};
  \node[anchor=south]       at (b) {$b$};
  \node[anchor=south west] at (c) {$c$};
  \node[anchor=north]      at (d) {$d$};
  \node[anchor=north]      at (e) {$e$};
  \node[anchor=south west] at (f) {$f$};

\end{tikzpicture}
    \subcaption{Subgraph of $H_6$}
    \label{Fig: Forbidden 4-6b }
    \end{subfigure}
    \caption{Depiction of the embedded contact subgraphs of $H_4,H_5,$ and $H_6$, the dashed line segment represents a degree of freedom}
    \label{Fig: Forbidden 4-6 }
\end{figure}
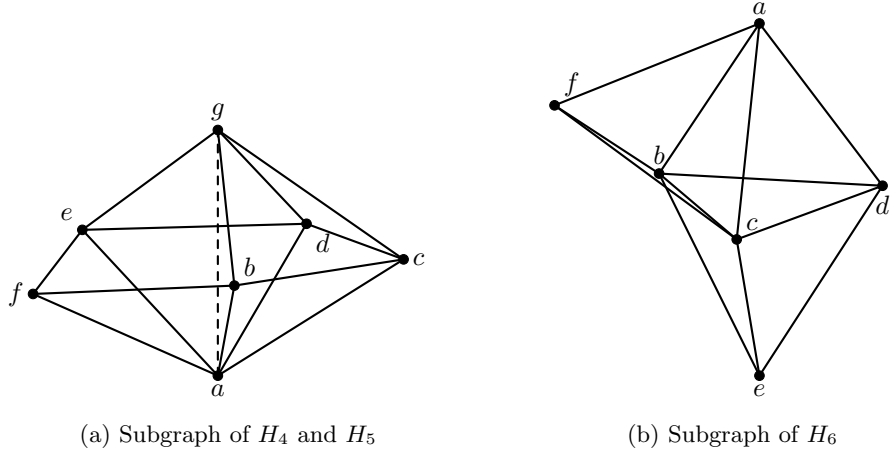

\begin{lemma}[Forbidding $H_4$ and $H_5$]\label{Lem: Forbidding 4 and 5}
    Let $a=(0,0,0)$, and $g=(0,0,s)$ for some $s\geq 1$. Suppose $b,c,d,$ and $e$ are at a unit distance from both $a$ and $g$. Furthermore, suppose $b$ is at a unit distance from $c$, $c$ is at a unit distance from $d$, and $d$ is at a unit distance from $e$. Additionally suppose $f$ is a point at a unit distance from $a,b,$ and $e$. Suppose all unspecified distances between points are at least one (see Figure \ref{Fig: Forbidden 4-6a }). If a point $h$ is at a unit distance from $c,d,$ and $f$, then it is at a distance less than one from $a,b,e,$ or $g$. Moreover, if a point $h'$ is at a unit distance from $b,d,$ and $f$, then it is at a distance less than one from $a,c,e,$ or $g$.
\end{lemma}
\begin{proof}
    Given $s$, the configuration of $a,b,c,d,e,$ and $g$ is determined up to rotation around the $z$-axis, since $b,c,d,$ and $e$ lie on the circle centered at $(0,0,s/2)$ with radius $\rho=\sqrt{1-\frac{s^2}{4}}$ on the perpendicular bisector of the line segment connecting $a$ to $g$. Since $b,c,d,$ and $e$ are consecutively a unit distance apart, they are consecutively at an angle of $\delta$ on this circle where $\sin(\frac{\delta}{2})=\frac{1}{2\rho}$. Suppose without loss of generality their coordinates are $(\rho \cos(k\delta),\rho \sin(k\delta),s/2)$ where $k=0,1,2,3$. Since $\sin(\frac{\delta}{2})=\frac{1}{2\rho}$ we obtain using the double-angle identities that $\cos(\delta)=1-\frac{1}{2\rho^2}$ and $\sin(\delta)=\frac{1}{\rho}\sqrt{1-\frac{1}{4\rho^2}}$. Using the double-angle formula again and the triple-angle formula we obtain $b=(\rho,0,s/2)$, $c=(\rho \cos(\delta),\rho \sin(\delta),s/2)$, $d=(\rho (2\cos(\delta)^2-1),2\rho \sin(\delta)\cos(\delta),s/2)$, and $e=(\rho (4\cos(\delta)^3-3\cos(\delta)),\rho (3\sin(\delta)-4\sin(\delta)^3),s/2)$.
    
    Since $f$ is a unit distance from $a,b,$ and $e$, it can be one of two points. In conclusion, for each $s$ there are at most two configurations of the points $a$ through $g$ up to rotation around the $z$-axis. 

    We now show $s\leq \sqrt{2}$. 
    We first note that $\delta\leq 90^\circ$ as the angular distance between $b$ and $c$, $c$ and $d$, and $d$ and $e$ are each $\delta$, and the angular distance between $b$ and $e$ is at least $\delta$ as $\|b-e\|\geq 1$. Since $\sin(\frac{\delta}{2})=\frac{1}{2\rho}$ we obtain $\rho\geq \frac{1}{\sqrt{2}}$ which implies $s\leq \sqrt{2}$.

    We again use interval arithmetic to verify both parts of the claim, splitting up the domain $s\in [1,\sqrt{2}]$ into $20\,000$ cells. We then compute the range of coordinates $b,c,d,$ and $e$ can have. As mentioned above there are two locations $f$ can be in. For each one we check if the constraint that $f$ is a distance at least one away from $c,d,$ and $g$ is possible. For each solution of $f$ that is possible, we compute the range $h$ and $h'$ can be in. There is only one such location for each as $a$ is adjacent to $b,c,d,$ and $e$. For each such location of $h$ and $h'$ we verify that one of the points $a$ through $g$ is at a distance less than one from $h$ or $h'$, which completes the proof.
\end{proof}

\begin{lemma}[Forbidding $H_6$]\label{Lem: Forbidding 6}
    Let $a,b,c,$ and $d$ be the vertices of a regular unit-length tetrahedron. Suppose $e$ is at a unit distance from $b,c,$ and $d$. Suppose $f$ is at a unit distance from $a,b,$ and $c$. 
    Suppose all unspecified distances between points are at least one (see Figure \ref{Fig: Forbidden 4-6b }). 
    Then any point at a unit distance from $c,e,$ and $f$ is at distance less than one from $a,b,$ or $d$.
\end{lemma}
    \begin{proof}
    Let $a$ through $f$ be as stated in the lemma. The configuration is determined up to translation and rotation and is depicted in Figure \ref{Fig: Forbidden 4-6b }. 
    Suppose $g$ is at a unit distance from $c,e,$ and $f$. Then $g$ can be one of two points (one of which is $b$). We use interval arithmetic to verify both points are at a distance less than one from $a,b,$ or $d$.
    \end{proof}

\begin{figure}
    \centering
    \begin{subfigure}[t]{0.45\textwidth}
\tdplotsetmaincoords{80}{100}
\begin{tikzpicture}[tdplot_main_coords, scale=3,
                    line cap=round, line join=round]

  \coordinate (a) at ( 0.8507,  0,       0);
  \coordinate (b) at ( 0.2629,  0.8090,  0);
  \coordinate (c) at (-0.6882,  0.5000,  0);
  \coordinate (d) at (-0.6882, -0.5000,  0);
  \coordinate (e) at ( 0.2629, -0.8090,  0);

  \coordinate (f) at (0, 0,  0.5258);
  \coordinate (g) at (0, 0, -0.5258);

  \draw[thick] (a) -- (b);
  \draw[thick] (b) -- (c);
  \draw[thick] (c) -- (d);
  \draw[thick] (d) -- (e);
  \draw[thick] (e) -- (a);

  \draw[thick] (f) -- (a);
  \draw[thick] (f) -- (b);
  \draw[thick] (f) -- (c);
  \draw[thick] (f) -- (d);
  \draw[thick] (f) -- (e);

  \draw[thick] (g) -- (a);
  \draw[thick] (g) -- (b);
  \draw[thick] (g) -- (c);
  \draw[thick] (g) -- (d);
  \draw[thick] (g) -- (e);

  \foreach \v in {a,b,c,d,e,f,g} \filldraw (\v) circle (0.6pt);

  \node[anchor=south west]       at (a) {$a$};
  \node[anchor=south west] at (b) {$b$};
  \node[anchor=south] at (c) {$c$};
  \node[anchor=north east] at (d) {$d$};
  \node[anchor=north] at (e) {$e$};
  \node[anchor=south]      at (f) {$f$};
  \node[anchor=north]      at (g) {$g$};

\end{tikzpicture}
\subcaption{subgraph of $H_7$}
\label{Fig: Forbidden 7-9a}
    \end{subfigure}
    \hfill
    \begin{subfigure}[t]{0.5\textwidth}

\tdplotsetmaincoords{80}{148}
\begin{tikzpicture}[tdplot_main_coords, scale=3,
                    line cap=round, line join=round]

  \coordinate (f) at (0, 0, 0);
  \coordinate (g) at (0, 0, 1);

  \coordinate (a) at ( 0.6738,  0.5446,  0.5);
  \coordinate (b) at (-0.2887,  0.8165,  0.5);
  \coordinate (c) at (-0.8660,  0,       0.5);
  \coordinate (d) at (-0.2887, -0.8165,  0.5);
  \coordinate (e) at ( 0.6738, -0.5446,  0.5);

  \coordinate (h) at ( 0.9571,  0,      -0.2898);
  \coordinate (i) at ( 0.9571,  0,       1.2898);
  
  \draw[thick] (g) -- (f);

  \draw[thick] (a) -- (b);
  \draw[thick] (b) -- (c);
  \draw[thick] (c) -- (d);
  \draw[thick] (d) -- (e);

  \draw[thick] (f) -- (a);
  \draw[thick] (f) -- (b);
  \draw[thick] (f) -- (c);
  \draw[thick] (f) -- (d);
  \draw[thick] (f) -- (e);

  \draw[thick] (g) -- (a);
  \draw[thick] (g) -- (b);
  \draw[thick] (g) -- (c);
  \draw[thick] (g) -- (d);
  \draw[thick] (g) -- (e);

  \draw[thick] (h) -- (a);
  \draw[thick] (h) -- (e);
  \draw[thick] (h) -- (f);

  \draw[thick] (i) -- (a);
  \draw[thick] (i) -- (e);
  \draw[thick] (i) -- (g);

  \foreach \v in {a,b,c,d,e,f,g,h,i} \filldraw (\v) circle (0.6pt);

  \node[anchor=south west] at (a) {$a$};
  \node[anchor=west]       at (b) {$b$};
  \node[anchor=west]       at (c) {$c$};
  \node[anchor=south east]      at (d) {$d$};
  \node[anchor=east] at (e) {$e$};
  \node[anchor=north] at (f) {$f$};
  \node[anchor=south] at (g) {$g$};
  \node[anchor=east]       at (h) {$h$};
  \node[anchor=east]       at (i) {$i$};

\end{tikzpicture}
    \subcaption{Subgraph of $H_9$}
    \label{Fig: Forbidden 7-9b}
    \end{subfigure}
    \caption{Depiction of the embedded contact subgraphs of $H_7$ and $H_9$}
    \label{Fig: Forbidden 7-9}
\end{figure}
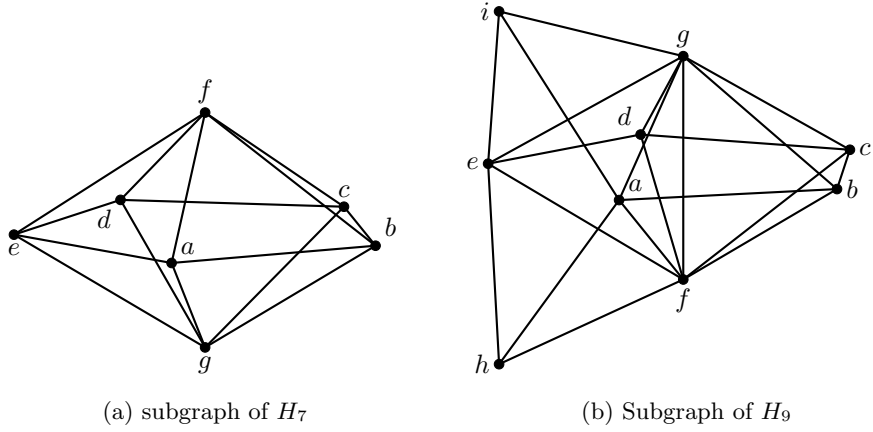

\begin{lemma}[Forbidding $H_7$]\label{Lem: Forbidding 7}
    Suppose $a,b,c,d,$ and $e$ form a cycle of unit distances and each of these points is at a unit distance from both $f$ and $g$.  Suppose all unspecified distances between points are at least one (see Figure \ref{Fig: Forbidden 7-9a}). Then any point at a unit distance from $a,c,$ and $f$ is at a distance less than one from $b,d,e,$ or $g$.
\end{lemma}
\begin{proof}
    The points $a$ through $e$ must lie on the circle at a unit distance from both $f$ and $g$, which implies the points $a$ through $e$ must be the vertices of a regular pentagon. 
    This implies the configuration of the points $a$ through $g$ is determined up to translation and rotation. 
    Without loss of generality the circle at a unit distance from both $f$ and $g$ is centered at the origin and lies in the $z=0$ plane.
    Then, up to rotation about the $z$ axis, the coordinates of $a$ through $e$ are $(\frac{1}{2\sin(\uppi/5)}\cos(\theta),\frac{1}{2\sin(\uppi/5)}\sin(\theta),0)$ for $\theta=\frac{2\uppi k}{5}$ where $k=0,1,2,3,$ and $4$. 
    The coordinates of $f$ and $g$ are $(0,0,\pm \sqrt{1-(\frac{1}{2\sin(\uppi/5)})^2})$. 
    We again use interval arithmetic to verify the two points at a unit distance from $a,c,$ and $f$ (one of which is $b$) are at a distance less than one from $b,d,e,$ or $g$.
\end{proof}

\begin{lemma}[Forbidding $H_8$]\label{Lem: Forbidding 8}
Suppose $a,b,c,$ and $d$ form a cycle of unit distances and each of these points is at a unit distance from both $e$ and $f$. Suppose $g$ is at a unit distance from $a,b,$ and $f$. Suppose all unspecified distances between points are at least one. Then any point at a unit distance from $a,b,$ and $f$, is at a distance less than one from $c,d,e,$ or $g$.
\end{lemma}
\begin{proof}
    A similar argument to the previous lemma shows that the vertices $a$ through $f$ form the vertex set of an octahedron. 
    Without loss of generality suppose the coordinates of $a$ through $d$ are $(\pm 1/2,\pm1/2,0)$ and the coordinates of $e$ and $f$ are $(0,0,\pm 1/\sqrt{2})$. We may assume $a=(1/2,1/2,0),b=(-1/2,1/2,0),$ and $f=(0,0,-1/\sqrt{2})$. There are two points at a unit distance from $a,b,$ and $f$ (one of which is $g$), so it suffices to show one of these points is at a distance less than one from $c,d,$ or $e$.
    Again we do this using interval arithmetic which completes the proof.
\end{proof}

\begin{lemma}[Forbidding $H_9$]\label{Lem: Forbidding 9}
    Suppose $a,b,c,d,$ and $e$ are consecutively a unit distance apart. Suppose $f$ and $g$ are a unit distance apart and also a unit distance away from $a,b,c,d,$ and $e$. Suppose $h$ is a unit distance away from $a,e,$ and $f$. Suppose $i$ is a unit distance away from $a,e,$ and $g$. Suppose all unspecified distances between points are at least one (see Figure \ref{Fig: Forbidden 7-9b}). Then $i$ and $h$ are at a distance greater than one from each other. 
\end{lemma}
\begin{proof}
    Suppose $f,g=(0,0,\pm1/2)$, then $a$ through $e$ lie on a circle centered at the origin with radius $\sqrt{3}/2$ in the plane $z=0$. The points $a$ through $e$ are consecutively separated by the angle $\delta=\arccos(1/3)$. Without loss of generality suppose $a$ through $e$ have the coordinates $(\sqrt{3}/2 \cos(k\delta),\sqrt{3}/2 \sin(k\delta),0)$ for $k=0,1,\dots ,4$. Since $h$ and $i$ are both a unit distance from three points already positioned, they can each be one of two points. However, since $f$ and $g$ are also at a unit distance from the same three points as $i$ and $h$ respectively, there is only one position for both $h$ and $i$. It suffices to verify that if $h$ and $i$ are a distance at least one from each of the points $a$ through $g$, then the distance between $h$ and $i$ is greater than one. We do this through interval arithmetic which completes the proof. 
\end{proof}

With the forbidden subgraphs established we can now prove Theorems \ref{The: Small Packings} and \ref{The: Enumeration}. 
For $n\leq 5$ all extremal contact graphs are known and their embeddings are unique up to translation and rotation.
For $n\leq 4$ they are complete graphs and for $n=5$ the contact graph is $K_5$ minus an edge. 
We prove $c(n)=3n-6$ for $n=6,7,8,$ and $9$ in that order and for each $n$, excluding $n=9$, we generate a list that contains every extremal contact graph on $n$ vertices.

Suppose we want to prove $c(n)=3n-6$ for some $n=6,7,8,$ or $9$ and we have proven $c(m)=3m-6$ for each $m$ such that $5\leq m<n$. 
Additionally suppose we have created lists, $\mathcal{G}_m$, that contain all extremal $m$-vertex contact graphs for each $m$ such that $5\leq m<n$.
We prove $c(n)=3n-6$ by contradiction and suppose $c(n)>3n-6$.
Then every extremal $n$-vertex contact graph must have the property that if the removal of $k\leq n-5$ vertices from the graph results in a graph with at least $3(n-k)-6$ edges, it must be in the list $\mathcal{G}_{n-k}$.
Otherwise we reach a contradiction as follows.
Since induced subgraphs of contact graphs are again contact graphs, if the removal of $k$ vertices resulted in a graph with greater than $3(n-k)-6$ edges, this contradicts that $c(n-k)=3(n-k)-6$, and in the case the removal of $k$ vertices resulted in a graph with exactly $3(n-k)-6$ edges that is not one of the graphs in the list $\mathcal{G}_{n-k}$, this contradicts that $\mathcal{G}_{n-k}$ contains every $(n-k)$-vertex extremal contact graph. 

We then search for all $n$-vertex graphs with greater than $3n-6$ edges that have the following properties which are necessary for it to be a contact graph on $n$ unit-diameter balls: 

\begin{enumerate}
    \item The graph must not have $K_5,K_{3,3},H_1,H_2,\dots, H_9$ as subgraphs.
    \item  The common neighborhood of any pair of vertices has at most five vertices.
    \item The common neighborhood of any adjacent pair of vertices must induce a subgraph of $P_5$.
    \item If the removal of $k$ vertices, where $k\leq n-6$, results in a graph with at least $3(n-k)-6$ edges, it must be one of the graphs listed in $\mathcal{G}_{n-k}$ (note in the implementation we only use this property for $k=1$ and $2$, the $k=1$ case imposes a minimum degree condition on our graph).
\end{enumerate}
 
We implement this graph search using SageMath. The code is attached to the arXiv version of this paper under the file name \texttt{graph\_searches.sage}. 

For each $n$ we find that no such graph satisfies these conditions, which contradicts $c(n)>3n-6$ and proves $c(n)\leq3n-6$. This together with the bound $c(n)\geq 3n-6$ for $n\geq 4$ from Theorem \ref{The: Rigid} proves $c(n)=3n-6$.

Now to generate $\mathcal{G}_n$ we search all $n$-vertex graphs that have $3n-6$ edges and satisfy the above four properties. We then add each of these graphs to the list $\mathcal{G}_n$. We find that $\mathcal{G}_6 ,\mathcal{G}_7$, and $\mathcal{G}_8$ have $2,5,$ and $13$ graphs respectively. These match the number of graphs in the candidate enumerations found in \cite{EmpAMB2011} which implies each graph is realizable as the contact graph of unit-diameter balls and hence the lists $\mathcal{G}_6 ,\mathcal{G}_7$, and $\mathcal{G}_8$ enumerate all extremal contact graphs on $6,7,$ and $8$ vertices.

\section{The contact number of packings on the FCC lattice}
In this section we prove Theorem \ref{The: FCC Upper Bound}, Theorem \ref{The: FCC asymptotics}, and Theorem \ref{The: Lower Bound}. 
To do this we transform the problem of determining $c_{A}(n)$ into an isoperimetric problem on a particular Cayley graph. 

\begin{definition}
    Let \(U\) be a finite set that generates \(\mathbb{Z}^d\) as a group and does not contain the identity.
    The (directed) Cayley graph denoted by $\mathbb{Z}^d_U$ is the graph with vertex set \(\mathbb{Z}^d\) where two vertices $u$ and $v$ are connected by an edge whenever $v-u\in U$.
    When \(U\) is symmetric (that is, $-u\in U$ for all $u\in U$), we consider \(\mathbb{Z}^d_U\) to be an undirected graph.
\end{definition}

Consider a packing of unit-diameter balls, where the center point of each ball lies on the face-centered cubic lattice $A_3$. 
Let $S\subseteq A_3$ be the set of center points in the packing. 
Two balls are in contact with each other if and only if their center points are at a unit distance from each other. 
Each element of $A_3$ has $12$ points at a unit distance from it. 
If we connect each point in $A_3$ by an edge to the 12 points at a unit distance away from it, we form the Cayley graph $A_3^U$ with vertex set $A_3$ and generated by the set of six pairs of antipodal unit vectors $U=\{( \pm \frac{1}{\sqrt{2}},\pm\frac{1}{\sqrt{2}},0), ( \pm \frac{1}{\sqrt{2}},0,\pm\frac{1}{\sqrt{2}}),(0, \pm \frac{1}{\sqrt{2}},\pm\frac{1}{\sqrt{2}})\}$. 
$A_3$ is clearly seen to be isomorphic to $\mathbb{Z}^3$. 
So, by taking the image of $U$ under this isomorphism, $A_3^U$ is isomorphic to a Cayley graph on $\mathbb{Z}^3$. 
Since $S\subseteq A_3$ and two balls are in contact with each other if and only if their centers are at a unit distance from each other, this implies that the contact graph of the packing is an induced subgraph of $A_3^U$, specifically $A_3^U[S]$.
Consequently, any upper bound on the number of edges in an $n$-vertex induced subgraph of $A_3^U$ is also an upper bound on $c_{A}(n)$.

To this end, consider an $n$-vertex subset $S$ of $A_3^U$ and the induced subgraph $A_3^U[S]$ with $e$ edges.
We define the edge boundary $\partial(S)$ of the subset $S$ to be the set of edges in $A_3^U$ where one vertex incident with this edge is in $S$ and the other is not in $S$. 
\[ \partial(S):=\{uv\in E(A_3^U): u\in S, v\notin S\}\]
Since $A_3^U$ is a $12$-regular graph, we can relate the number of edges of $A_3^U[S]$ to its edge boundary. 
For any vertex $v\in S$, define $\partial(v):=\{uv \in E(A_3^U): u\notin S\}$.
The degree of the vertex $v$ is equal to $12$ minus $|\partial(v)|$.
Each edge in $\partial(v)$ contributes one element to $\partial(S)$, giving us $|\partial(S)|=\sum_{v \in S} |\partial(v)|$, which implies that 
 
\[2e=\sum_{v\in S} \deg(v)=\sum_{v\in S} (12-|\partial(v)|)=12n-|\partial(S)|.\]

So 
\begin{align}\label{Eq: Edge to Edge Boundary}
e=6n-\frac{|\partial(S)|}{2}.
\end{align}

With \eqref{Eq: Edge to Edge Boundary}, we see that the problem of maximizing the number of edges over all $n$-vertex induced subgraphs of $A_3^U$ is equivalent to minimizing the edge boundary over all $n$-vertex subsets of $A_3^U$. 
Since the former is equivalent to determining $c_{A}(n)$, to complete the proof of Theorem \ref{The: FCC Upper Bound} and Theorem \ref{The: FCC asymptotics}, we will show the relevant lower bounds on the edge boundary of any $n$-vertex subset of $A_3^U$.

\subsection{The asymptotics of the contact number on the FCC lattice}

In this section we prove Theorem \ref{The: FCC asymptotics} by using \eqref{Eq: Edge to Edge Boundary} and a theorem of Barber and Erde in \cite{BarErd2018} regarding the edge boundary of subsets of Cayley graphs on integer lattices.
First we need some notation. For two points $x,y\in \mathbb{R}^d$ define $[x,y]$ to be the closed line segment connecting $x$ and $y$.
Given two sets $A,B\subseteq\mathbb{R}^d$ their Minkowski sum is $A+B=\{a+b: a\in A, b\in B\}$.

\begin{theorem}[Barber, Erde]\label{The: Barber Erde}
    Let $U=\{u_1,\dots, u_k\}$ be a finite set of non-zero vectors that generate $\mathbb{Z}^d$ as a group and let $Z=[0,u_1]+[0,u_2]+\dots +[0,u_k]$, then
    \[ 
    \min_{S\subseteq \mathbb{Z}^d: |S|=n}|\partial(S)|=(1+o(1))d\vol{Z}^\frac{1}{d}n^{1-\frac{1}{d}}.  
    \]
    This minimum is witnessed by intersections of scaled copies of $Z$ with $\mathbb{Z}^d$.
\end{theorem}

To apply this to our situation, we only need to calculate the volume of $Z$. 
Since the theorem applies only to Cayley graphs on $\mathbb{Z}^d$, we accomplish this by first calculating the volume of $Z'=\sum_{u\in U}[0,u]$, where $U$ is the generating set of $A_3^U$. 
We then calculate the determinant of the matrix, $T$, which forms an isomorphism between $\mathbb{Z}^3$ and $A_3$. 
This determinant will relate the volume of $Z'$ with the volume of $Z=\sum_{u\in U} [0,T^{-1}(u)]$, which we can then plug into Theorem \ref{The: Barber Erde}.
Since $A_3^U$ is isomorphic to the Cayley graph on $\mathbb{Z}^3$ generated by $\{T^{-1}(u): u\in U\}$, the bound we obtain on the minimum edge boundary of the latter will also apply to the former.

The linear transformation taking $\mathbb{Z}^3$ to $A_3$ is 
\[T=\begin{pmatrix}
    \frac{1}{\sqrt{2}}&-\frac{1}{\sqrt{2}}&0\\
    \frac{1}{\sqrt{2}}&\frac{1}{\sqrt{2}}&\frac{1}{\sqrt{2}}\\
    0&0&\frac{1}{\sqrt{2}} \\
    \end{pmatrix}\]

which has a determinant of $\frac{1}{\sqrt{2}}$. 
We can relate the volume of $Z$ and $Z'$ by the following 

\begin{align}\label{Eq: Z to Z'}
\vol{\sum_{u\in U}[0,T^{-1}(u)]}=\frac{1}{\det(T)} \vol{\sum_{u\in U}[0,u]}.
\end{align}

It is a straightforward calculation to see that $\sum_{u\in U}[0,u]$ is equal to the truncated octahedron with an edge length of $2$. 
This truncated octahedron is equal to the convex hull of all coordinate permutations of $(\pm\sqrt2,0,\pm 2\sqrt2)$.
Its volume is therefore equal to $64\sqrt{2}$.
By \eqref{Eq: Z to Z'} this implies $\vol{Z}=128$ which by the result of Barber and Erde gives us
\[ \min_{S\subseteq A_3: |S|=n}|\partial(S)|= (1+o(1))12\sqrt[3]{2}n^\frac{2}{3}.\]

By \eqref{Eq: Edge to Edge Boundary} this implies 
\[
c_{A}(n)=\max_{S\subseteq A_3: |S|=n} 6n-\frac{|\partial(S)|}{2}= 6n-(1+o(1))6\sqrt[3]{2}n^\frac{2}{3}.
\]

\subsection{An upper bound for the contact number on the FCC lattice}

In this section we aim to prove Theorem \ref{The: FCC Upper Bound}.
To this end, it suffices to show that for any $n$, any $n$-vertex subset $S$ of $A_3^U$ has an edge boundary with cardinality at least $\frac{12}{\sqrt[6]{2}}n^\frac{2}{3}$.
Let $S$ be an $n$-element subset of $A_3^U$.
To find a lower bound on $|\partial(S)|$, we find a lower bound on the cardinality of a subset of $\partial(S)$ called the outer edge boundary of $S$, which is denoted by $\partial^*(S)$. 
For $u, v\in \mathbb{R}^3$, denote by $\overrightarrow{uv}$ the ray emanating from $u$ in the direction $v$. The outer edge boundary of $S$ is 
\[\partial^*(S):=\{uv\in E(A_3^U): u\in S, v\notin S, \overrightarrow{uv}\cap S=\{u\}\}.\]
It is immediate that $\partial^*(S)\subseteq \partial(S)$. 
It is also clear that for every $uv\in \partial^*(S)$ there is another edge $xy\in \partial^*(S)$ that is contained in the line through $u$ and $v$ (note that $u$ and $x$ need not be distinct).
Moreover, every line parallel to a direction in $U$ that intersects an element of $S$ will contain exactly two elements in the outer edge boundary.
Thus the cardinality of the outer edge boundary of $S$ is equal to twice the number of lines, parallel to a direction in $U$, that intersect $S$.
We can count the number of lines parallel to a direction $u\in U$ that intersect $S$ by counting the cardinality of the orthogonal projection of $S$ onto the $2$-dimensional subspace $E$ orthogonal to $u$. 
There are $6$ projections in total, as the generating set $U$ has $6$ antipodal pairs of vectors.

In summary, the cardinality of the outer edge boundary of $S$ is the sum, over the $6$ subspaces orthogonal to a direction in $U$, of twice the cardinality of $S$ when projected to this subspace.
To bound the cardinality of the projections in terms of the cardinality of $S$ we use the following inequality which can be found in \cite{Brascamp}.

 \begin{theorem}[Brascamp-Lieb, Ball, Barthe]\label{The: Brascamp-lieb}
    For a linear subspace $E$ of $\mathbb{R}^d$ let $P_{E}$ denote the orthogonal projection of $\mathbb{R}^d$ onto $E$.
    Suppose $E_1,E_2,\dots, E_k$ are linear subspaces of $\mathbb{R}^d$ and $p_1,p_2,\dots, p_k$ are positive numbers satisfying $\sum_{i=1}^k p_i P_{E_i} =I$ where $I$ denotes the $d\times d$ identity matrix.
    For non-negative $f_i \in L_1(E_i)$ the following inequality holds. 
    \[
    \int_{\mathbb{R}^d}\prod_{i=1}^k f_i( P_{E_i}(x))^{p_i}dx \leq \prod_{i=1}^k \left( \int_{E_i} f_i(x) dx\right)^{p_i}
    \]
\end{theorem}

 For $i=1,2,\dots, 6$ let $E_i$ denote the 2-dimensional subspace of $\mathbb{R}^3$ orthogonal to $u_i$, where 
 \begin{align*}
 u_1,u_2,\dots, u_6&=\frac{1}{\sqrt{2}}(1,0,1),\frac{1}{\sqrt{2}}(0,1,1),\frac{1}{\sqrt{2}}(1,-1,0),
 \\ &\mathrel{\phantom{=}}\frac{1}{\sqrt{2}}(1,0,-1),\frac{1}{\sqrt{2}}(0,1,-1),\frac{1}{\sqrt{2}}(1,1,0).
 \end{align*}
 The first claim is that $\sum_{i=1}^6 \frac{1}{4} P_{E_i} =I$. 
 For any $i=1,2,\dots, 6$, since $u_i$ is a unit vector, the projection onto $E_i$ is $P_{E_i}=I-P_{u_i}$, where $P_{u_i}$ denotes the orthogonal projection onto $\mathrm{span}(u_i)$. 
 So for any $x\in \mathbb{R}^3$
 \[\sum_{i=1}^6 \frac{1}{4} P_{E_i}(x)=\frac{1}{4} \sum_{i=1}^6 \left(I(x)-P_{u_i}(x)\right)=\frac{1}{4}(6I(x)-\sum_{i=1}^6 \langle u_i,x\rangle u_i ).\]
 It is a straightforward calculation to see that 
 \[
 \sum_{i=1}^6 \langle u_i,x\rangle u_i=2x.
 \]

So we obtain that \(\sum_{i=1}^6 \frac{1}{4} P_{E_i}=\frac{1}{4}(6I-2I)=I\) as claimed.
Next we need to define $f_i: E_i \rightarrow \mathbb{R}$, which will relate the cardinality of $S$ to the cardinality of $S_i:=P_{E_i}(S)$.
We accomplish this by placing a small translate of a polytope $P$, to be determined later, around each point of $S$.
We make $P$ small enough so that $S+P$ forms a packing of translates of $P$ in $\mathbb{R}^3$.
Since $S+P$ is a packing, we have $|S|\vol P=\vol {S+P}$.
Now we will define $f_i$ to be the indicator function of $P_{E_i}(S+P)$, that is, $\chi_{P_{E_i}(S+P)}$, which is non-negative and an element of $L_1(E_i)$. 
Notice additionally that $\chi_{S+P}(x)\leq\prod_{i=1}^6\chi_{P_{E_i}(S+P)}(P_{E_i}(x))$.

From these observations we obtain
\begin{align*}
    |S|\vol{P}&=
    \vol{S+P}=
    \int_{\mathbb{R}^3}\chi_{S+P}(x) dx \leq \int_{\mathbb{R}^3}\prod_{i=1}^6\left( \chi_{P_{E_i}(S+P)}(P_{E_i}(x))\right)^{\frac{1}{4}}dx\\
    &\leq \prod_{i=1}^6\left(\int_{E_i}\chi_{P_{E_i}(S+P)}(x)dx\right)^{\frac{1}{4}}=
    \prod_{i=1}^6 \area{P_{E_i}(S+P)}^{\frac{1}{4}} \\
    &\leq \prod_{i=1}^6\Bigl( \area{P_{E_i}(P)}|S_i|\Bigr)^{\frac{1}{4}} = \left(\prod_{i=1}^6 \area{P_{E_i}(P)}^\frac{1}{6}\right)^\frac{6}{4} \left(\prod_{i=1}^6|S_i|^\frac{1}{6}\right)^{\frac{6}{4}}\\
    &\leq \left(\frac{1}{6}\sum_{i=1}^6 \area{P_{E_i}(P)}\right)^\frac{3}{2} \left(\frac{1}{6}\sum_{i=1}^6|S_i|\right)^\frac{3}{2} \\
    &=\left(\frac{1}{6}\sum_{i=1}^6 \area{P_{E_i}(P)}\right)^\frac{3}{2} \left(\frac{1}{12}|\partial^*(S)|\right)^\frac{3}{2}\\
    &=
    (\frac{1}{72})^\frac{3}{2}\left(\sum_{i=1}^6 \area{P_{E_i}(P)}\right)^\frac{3}{2}|\partial^*(S)|^\frac{3}{2} \\
    &\leq (\frac{1}{72})^\frac{3}{2} |\partial(S)|^\frac{3}{2}\left(\sum_{i=1}^6 \area{P_{E_i}(P)}\right)^\frac{3}{2}.
\end{align*}

Here, the second inequality is by Theorem \ref{The: Brascamp-lieb}, the fourth is the AM-GM inequality applied twice, and the fifth is that the outer edge boundary is a subset of the edge boundary.

All together we obtain 
\begin{align}\label{Eq: Brascamp application}
    \frac{72|S|^\frac{2}{3}\vol{P}^\frac{2}{3}}{\sum_{i=1}^6 \area{P_{E_i}(P)}}\leq |\partial(S)|.
\end{align}

In order to obtain the best lower bound using \eqref{Eq: Brascamp application}, we need to choose a polytope $P$ that, for a fixed volume, minimizes the sum of the areas of the projections of $P$ onto each $E_i$. 
The best polytope, unsurprisingly, turns out to be the truncated octahedron.
To see this, consider the following

\begin{align*}
\sum_{i=1}^6 \area{P_{E_i}(P)}&=\sum_{i=1}^6\lim_{\epsilon\rightarrow0^+}\frac{\vol{P+\epsilon[0,u_i]}-\vol{P}}{\epsilon}\\
&=\lim_{\epsilon\rightarrow0^+}\sum_{i=1}^6\frac{\vol{P+\epsilon[0,u_i]}-\vol{P}}{\epsilon}\\
&=\lim_{\epsilon\rightarrow0^+}\sum_{i=1}^6\frac{\vol{P+\epsilon\sum_{k=1}^i[0,u_k]}-\vol{P+\epsilon\sum_{k=1}^{i-1}[0,u_k]}}{\epsilon}\\
&=\lim_{\epsilon\rightarrow0^+}\frac{\vol{P+\epsilon\sum_{k=1}^6[0,u_k]}-\vol{P}}{\epsilon}.
\end{align*}
We now apply Minkowski's first inequality.
\begin{theorem}[Minkowski's first inequality found in \cite{Gardner}]
    Let $K$ and $L$ be convex bodies in $\mathbb{R}^d$, then 
    \[ 
    \frac{1}{d} \lim_{\epsilon\rightarrow0^+}\frac{\vol{K+\epsilon L}-\vol{K}}{\epsilon}\geq \vol{K}^\frac{d-1}{d}\vol{L}^\frac{1}{d},
    \]
    with equality if and only if $K$ and $L$ are homothetic.
\end{theorem}
From this we obtain
\[
\frac{1}{3}\sum_{i=1}^6 \area{P_{E_i}(P)}\geq \vol{P}^\frac{2}{3}\vol{\sum_{k=1}^6[0,u_k]}^\frac{1}{3}
\]
with equality if and only if $P$ is homothetic to $\sum_{k=1}^6[0,u_k]$. 
Since $\sum_{k=1}^6[0,u_k]$ is the truncated octahedron, the bound \eqref{Eq: Brascamp application} is sharpest when $P$ is homothetic to the truncated octahedron.
Additionally, notice that the bound in \eqref{Eq: Brascamp application} is invariant under dilation of $P$, implying we can do the calculation of the volume and projections when the edge length of the truncated octahedron is $1$. The volume in this case is $8\sqrt{2}$. All six projections are congruent and each has an area of $4\sqrt{2}$. One such projection is depicted in Figure \ref{Fig: TruOct}. This implies
\[
|\partial(S)|\geq 12(2^{-\frac{1}{6}})|S|^{\frac{2}{3}}.
\]
By \eqref{Eq: Edge to Edge Boundary}, this implies Theorem \ref{The: FCC Upper Bound}.

\subsection{A lower bound for the contact number on the FCC lattice}

In this section we prove Theorem \ref{The: Lower Bound}, which will follow immediately after proving the following lemma.

\begin{lemma}\label{Lem: LB}
    Consider the truncated octahedron with side length $k-1$ positioned so that exactly $k$ lattice points of $A_3$ are contained in each of its edges.
    The number of lattice points of $A_3$ contained in the truncated octahedron equals
    \[
    16k^3-33k^2+24k-6.
    \]
    The number of edges in the induced subgraph of $A_3^U$ on the $n$ lattice points contained in the truncated octahedron is
    \[6n-6(8k^2-11k+4).\]
\end{lemma}
\begin{proof}
    We begin by observing that for any $a\geq2$ there are $\sum_{i=1}^a i^2=\frac{2a^3+3a^2+a}{6}$ lattice points of $A_3$ such that their convex hull is a square pyramid with an edge length of $a-1$ where each edge of the square pyramid has exactly $a$ lattice points on it.
    It follows that there are $\frac{2a^3+3a^2+a}{3}-a^2=\frac{2a^3+a}{3}$ lattice points of $A_3$ such that their convex hull is an octahedron with an edge length of $a-1$ and exactly $a$ lattice points on each edge.
    Letting $a=3k-2$, we want to compute how many lattice points are in the truncated octahedron with an edge length of $k-1$.
    This can be computed by removing $6$ square pyramids with an edge length of $k-2$ from the octahedron with an edge length of $a-1$. 
    It follows that the number of lattice points in the truncated octahedron with an edge length of $k-1$ and $k$ lattice points per edge is
    \[
    \frac{2(3k-2)^3+(3k-2)}{3}-(2(k-1)^3+3(k-1)^2+(k-1))=16k^3-33k^2+24k-6.
    \]
    For the second part of the lemma, notice that each vertex of $A_3^U$ in the relative interior of the truncated octahedron has a degree of $12$. A vertex on the boundary has degree $9$ if it lies in the relative interior of a hexagonal face, degree $8$ if in the relative interior of a square face, degree $7$ if in the relative interior of an edge, and degree $6$ if it is a vertex of the truncated octahedron.   
    
    In the $2$-dimensional section of the $A_3$ lattice containing a hexagonal face of the truncated octahedron, the $A_3$ lattice is the unit-length triangular lattice.
    A regular hexagon with $k$ lattice points per side contains $3(k-1)^2-3(k-1)+1$ points of the triangular lattice in its relative interior.
    In the 2-dimensional section of the $A_3$ lattice containing a square face, the $A_3$ lattice is isomorphic to the integer grid.
    A square with $k$ lattice points per side contains $(k-2)^2$ lattice points in its relative interior.
    There are $36$ edges of the truncated octahedron each with $k-2$ points in their relative interiors.
    Finally the truncated octahedron has $24$ vertices.

    With these observations made, we obtain
    \begin{align*}
        \sum_{v} \deg(v)&=12n-\sum_{v}|\partial(v)|\\
        &=12n-3\times8(3(k-1)^2-3(k-1)+1)\\
        &\mathrel{\phantom{=}} {} -4\times 6(k-2)^2-5\times 36(k-2)-6\times 24\\
        &=12n-12(8k^2-11k+4)
    \end{align*}
which completes the proof of the lemma.
\end{proof}

To complete Theorem \ref{The: Lower Bound}, if $n=16k^3-33k^2+24k-6$ then by solving for $k$ and using Lemma \ref{Lem: LB} we obtain that the number of edges in the subgraph of $A_3^U$ induced on the $n$ lattice points contained in the truncated octahedron is 
\begin{align*}
&\mathrel{\phantom{=}}6n-\frac{3}{16}  \left((16 \sqrt{64 n^2 - 13 n + 2} - 128 n + 13)^{\frac{2}{3}}\right. \\
&\qquad\qquad\qquad \left. \mbox{}+ \frac{49}{(16 \sqrt{64 n^2 - 13 n + 2} - 128 n + 13)^{\frac{2}{3}}} - 7\right)\\
&> 6n-\frac{3}{16}   \frac{49}{(16 \sqrt{64 n^2 - 13 n + 2} - 128 n + 13)^{\frac{2}{3}}} \\
&= 6n-\frac{3}{16}(16 \sqrt{64 n^2 - 13 n + 2} + 128 n - 13)^{\frac{2}{3}}
\\ &> 6n-\frac{3}{16}(32 \sqrt{64 n^2 - 13 n + 2})^{\frac{2}{3}} > 6n-\frac{3}{16}(32 \sqrt{64 n^2 })^{\frac{2}{3}}=6n-6\sqrt[3]{2}n^\frac{2}{3}.
\end{align*}

\section*{Acknowledgments}
I would like to thank Konrad Swanepoel for our helpful discussions and his insightful suggestions throughout the development of the results and the creation of this paper. I would also like to thank Mihir Neve for noticing a simplification in the last part of the proof of Theorem \ref{The: Rigid}. Lastly, Claude (Anthropic) was used to help with the implementation of \texttt{Small\_sphere\_packings.py} and \texttt{graph\_searches.sage}.

\bibliographystyle{plain}
\bibliography{main}

@incollection {BezKhan2018,
    AUTHOR = {Bezdek, K\'aroly and Khan, Muhammad A.},
     TITLE = {Contact numbers for sphere packings},
 BOOKTITLE = {New trends in intuitive geometry},
    SERIES = {Bolyai Soc. Math. Stud.},
    VOLUME = {27},
     PAGES = {25--47},
 PUBLISHER = {J\'anos Bolyai Math. Soc., Budapest},
      YEAR = {2018},
      ISBN = {978-3-662-57412-6; 978-3-662-57413-3},
   MRCLASS = {52C17 (52C10)},
  MRNUMBER = {3889255},
MRREVIEWER = {Peter\ G.\ Boyvalenkov},
}

@article {BezReid2013,
    AUTHOR = {Bezdek, K\'aroly and Reid, Samuel},
     TITLE = {Contact graphs of unit sphere packings revisited},
   JOURNAL = {J. Geom.},
  FJOURNAL = {Journal of Geometry},
    VOLUME = {104},
      YEAR = {2013},
    NUMBER = {1},
     PAGES = {57--83},
      ISSN = {0047-2468,1420-8997},
   MRCLASS = {05B40 (11H31 52C17 52C45)},
  MRNUMBER = {3047448},
       DOI = {10.1007/s00022-013-0156-4},
       URL = {https://doi.org/10.1007/s00022-013-0156-4},
}

@article {Bez,
    AUTHOR = {Bezdek, K\'aroly },
     TITLE = {Contact numbers for congruent sphere packings in {E}uclidean
              3-space},
   JOURNAL = {Discrete Comput. Geom.},
  FJOURNAL = {Discrete \& Computational Geometry. An International Journal
              of Mathematics and Computer Science},
    VOLUME = {48},
      YEAR = {2012},
    NUMBER = {2},
     PAGES = {298--309},
      ISSN = {0179-5376,1432-0444},
   MRCLASS = {52C17 (05B40 11H31 52A40)},
  MRNUMBER = {2946449},
MRREVIEWER = {Zsolt\ L\'angi},
       DOI = {10.1007/s00454-012-9405-9},
       URL = {https://doi.org/10.1007/s00454-012-9405-9},
}

@article {EmpAMB2011,
    AUTHOR = {Arkus, Natalie and Manoharan, Vinothan N. and Brenner, Michael
              P.},
     TITLE = {Deriving finite sphere packings},
   JOURNAL = {SIAM J. Discrete Math.},
  FJOURNAL = {SIAM Journal on Discrete Mathematics},
    VOLUME = {25},
      YEAR = {2011},
    NUMBER = {4},
     PAGES = {1860--1901},
      ISSN = {0895-4801,1095-7146},
   MRCLASS = {52C17 (65H10)},
  MRNUMBER = {2873224},
MRREVIEWER = {Philippe\ Ryckelynck},
       DOI = {10.1137/100784424},
       URL = {https://doi.org/10.1137/100784424},
}

@article {EmpHolmes-Cerfon2016,
    AUTHOR = {Holmes-Cerfon, Miranda C.},
     TITLE = {Enumerating rigid sphere packings},
   JOURNAL = {SIAM Rev.},
  FJOURNAL = {SIAM Review},
    VOLUME = {58},
      YEAR = {2016},
    NUMBER = {2},
     PAGES = {229--244},
      ISSN = {1095-7200,0036-1445},
   MRCLASS = {82D80 (52C25)},
  MRNUMBER = {3493944},
       DOI = {10.1137/140982337},
       URL = {https://doi.org/10.1137/140982337},
}

@article {BarErd2018,
    AUTHOR = {Barber, Ben and Erde, Joshua},
     TITLE = {Isoperimetry in integer lattices},
   JOURNAL = {Discrete Anal.},
  FJOURNAL = {Discrete Analysis},
      YEAR = {2018},
     PAGES = {Paper No. 7, 16},
      ISSN = {2397-3129},
   MRCLASS = {05C25 (05C63)},
  MRNUMBER = {3819052},
       DOI = {10.19086/da.3555},
       URL = {https://doi.org/10.19086/da.3555},
}

@article {BarErdeKevRob2023,
    AUTHOR = {Barber, Ben and Erde, Joshua and Keevash, Peter and Roberts,
              Alexander},
     TITLE = {Isoperimetric stability in lattices},
   JOURNAL = {Proc. Amer. Math. Soc.},
  FJOURNAL = {Proceedings of the American Mathematical Society},
    VOLUME = {151},
      YEAR = {2023},
    NUMBER = {12},
     PAGES = {5021--5029},
      ISSN = {0002-9939,1088-6826},
   MRCLASS = {05D99 (11P70)},
  MRNUMBER = {4648905},
MRREVIEWER = {Yuval\ Filmus},
       DOI = {10.1090/proc/16439},
       URL = {https://doi.org/10.1090/proc/16439},
}

@article {Gardner,
    AUTHOR = {Gardner, R. J.},
     TITLE = {The {B}runn-{M}inkowski inequality},
   JOURNAL = {Bull. Amer. Math. Soc. (N.S.)},
  FJOURNAL = {American Mathematical Society. Bulletin. New Series},
    VOLUME = {39},
      YEAR = {2002},
    NUMBER = {3},
     PAGES = {355--405},
      ISSN = {0273-0979,1088-9485},
   MRCLASS = {26D15 (52A40)},
  MRNUMBER = {1898210},
MRREVIEWER = {Apostolos\ A.\ Giannopoulos},
       DOI = {10.1090/S0273-0979-02-00941-2},
       URL = {https://doi.org/10.1090/S0273-0979-02-00941-2},
}

@article{Harborth74,
  title={Solution to problem {664 A}},
  author={Harborth, Heiko},
  journal={Elem. Math},
  volume={29},
  pages={14--15},
  year={1974}
}

@article {Bez2002ConvUB,
    AUTHOR = {Bezdek, K\'aroly},
     TITLE = {On the maximum number of touching pairs in a finite packing of
              translates of a convex body},
   JOURNAL = {J. Combin. Theory Ser. A},
  FJOURNAL = {Journal of Combinatorial Theory. Series A},
    VOLUME = {98},
      YEAR = {2002},
    NUMBER = {1},
     PAGES = {192--200},
      ISSN = {0097-3165,1096-0899},
   MRCLASS = {52C17},
  MRNUMBER = {1897933},
MRREVIEWER = {Marek\ Lassak},
       DOI = {10.1006/jcta.2001.3204},
       URL = {https://doi.org/10.1006/jcta.2001.3204},
}

@article {SCDW1953Kissing,
    AUTHOR = {Sch\"utte, K. and van der Waerden, B. L.},
     TITLE = {Das {P}roblem der dreizehn {K}ugeln},
   JOURNAL = {Math. Ann.},
  FJOURNAL = {Mathematische Annalen},
    VOLUME = {125},
      YEAR = {1953},
     PAGES = {325--334},
      ISSN = {0025-5831,1432-1807},
   MRCLASS = {52.0X},
  MRNUMBER = {53537},
MRREVIEWER = {H.\ S. M. Coxeter},
       DOI = {10.1007/BF01343127},
       URL = {https://doi.org/10.1007/BF01343127},
}

@misc{Brascamp,
      title={The {B}rascamp-{L}ieb inequality in {C}onvex {G}eometry and in the {T}heory of {A}lgorithms}, 
      author={Károly J. Böröczky},
      year={2024},
      eprint={2412.11227},
      archivePrefix={arXiv},
      primaryClass={math.MG},
      url={https://arxiv.org/abs/2412.11227}, 
}

@article {Hales2005,
    AUTHOR = {Hales, Thomas C.},
     TITLE = {A proof of the {K}epler conjecture},
   JOURNAL = {Ann. of Math. (2)},
  FJOURNAL = {Annals of Mathematics. Second Series},
    VOLUME = {162},
      YEAR = {2005},
    NUMBER = {3},
     PAGES = {1065--1185},
      ISSN = {0003-486X,1939-8980},
   MRCLASS = {52C17},
  MRNUMBER = {2179728},
MRREVIEWER = {G\'eza\ T\'oth},
       DOI = {10.4007/annals.2005.162.1065},
       URL = {https://doi.org/10.4007/annals.2005.162.1065},
}

@article {Arb,
    AUTHOR = {Johansson, Fredrik},
     TITLE = {Arb: efficient arbitrary-precision midpoint-radius interval
              arithmetic},
   JOURNAL = {IEEE Trans. Comput.},
  FJOURNAL = {Institute of Electrical and Electronics Engineers.
              Transactions on Computers},
    VOLUME = {66},
      YEAR = {2017},
    NUMBER = {8},
     PAGES = {1281--1292},
      ISSN = {0018-9340,1557-9956},
   MRCLASS = {99-03},
  MRNUMBER = {3681746},
       DOI = {10.1109/TC.2017.2690633},
       URL = {https://doi.org/10.1109/TC.2017.2690633},
}

@manual{Flint,
  key    = {{FLINT}},
  author = {The {FLINT} team},
  title  = {{FLINT}: {F}ast {L}ibrary for {N}umber {T}heory},
  year   = {2026},
  note   = {Version 3.6.0, \texttt{arb} module, \url{https://flintlib.org/doc/arb.html}}
}

@article {Erdos46,
    AUTHOR = {Erd\H{o}s, P.},
     TITLE = {On sets of distances of {$n$} points},
   JOURNAL = {Amer. Math. Monthly},
  FJOURNAL = {American Mathematical Monthly},
    VOLUME = {53},
      YEAR = {1946},
     PAGES = {248--250},
      ISSN = {0002-9890,1930-0972},
   MRCLASS = {48.0X},
  MRNUMBER = {15796},
       DOI = {10.2307/2305092},
       URL = {https://doi.org/10.2307/2305092},
}

@article {Erdos75,
    AUTHOR = {Erd\H{o}s, Paul},
     TITLE = {On some problems of elementary and combinatorial geometry},
   JOURNAL = {Ann. Mat. Pura Appl. (4)},
  FJOURNAL = {Annali di Matematica Pura ed Applicata. Serie Quarta},
    VOLUME = {103},
      YEAR = {1975},
     PAGES = {99--108},
      ISSN = {0003-4622},
   MRCLASS = {05B25},
  MRNUMBER = {411984},
MRREVIEWER = {N.\ G.\ de Bruijn},
       DOI = {10.1007/BF02414146},
       URL = {https://doi.org/10.1007/BF02414146},
}

\end{document}